\documentclass[a4paper,11pt]{article}
\usepackage{amsmath,amssymb,amsthm,amscd}
\usepackage{graphicx,mathrsfs}

\newtheorem{Thm}{Theorem}[section]
\newtheorem{Lem}[Thm]{Lemma}
\newtheorem{Cor}[Thm]{Corollary}
\newtheorem{Prop}[Thm]{Proposition}

\newtheorem{Def}[Thm]{Definition}
\newtheorem{Rem}[Thm]{Remark}

\makeatletter \@addtoreset{equation}{section} \makeatother

\newcommand{\R}{\mathbb{R}}
\newcommand{\B}{\mathbb{B}}
\newcommand{\C}{\mathbb{C}}

\newcommand{\N}{\mathbb{N}}
\newcommand{\dbar}{d\hspace*{-0.08em}\bar{}\hspace*{0.1em}}

\begin{document}
\title{The Mellin-Edge Quantisation for Corner Operators}
\author{B.-W. Schulze\footnote{B.-Wolfgang Schulze, Institute of Mathematics, University of Potsdam, 14469, Potsdam, Germany; e-mail: schulze@math.uni-potsdam.de}, Y. Wei\footnote{Yawei Wei, School of Mathematical Sciences and LPMC, Nankai University, 300071 Tianjin, China; e-mail: weiyawei@nankai.edu.cn}}

\pagestyle{headings}
\maketitle
\date{}

\begin{abstract}
We establish a quantisation of corner-degenerate symbols, here called Mellin-edge quantisation, on a manifold $M$ with second order singularities. The typical ingredients come from the ``most singular" stratum of $M$ which is a second order edge where the infinite transversal cone has a base $B$ that is itself a manifold with smooth edge. The resulting operator-valued amplitude functions on the second order edge are formulated purely in terms of Mellin symbols taking values in the edge algebra over $B.$ In this respect our result is formally analogous to a quantisation rule of \cite{Gil2} for the simpler case of edge-degenerate symbols that corresponds to the singularity order 1. However, from the singularity order 2 on there appear new substantial difficulties for the first time, partly caused by the edge singularities of the cone over $B$ that tend to infinity.\footnote{Acknowledgements: This article has been partially supported by the NSFC (National Science Foundation of China) under the grant 11001135, by the TSTC under the grant 10JCYBJC25200, and by the DFG (Deutsche Forschungsgemeinschaft) within the project ``Partial Differential Equations in Geometry and Mathematical Physics", moreover, by the Chern Institute of Mathematics in Tianjin, China, and by a Research Grant at the Nankai University.}
\end{abstract}

\tableofcontents
\newpage
\section*{Introduction}
\addcontentsline{toc}{section}{Introduction}
      \markboth{INTRODUCTION}{INTRODUCTION}
In this paper we study the analysis of pseudo-differential operators $A$ on a corner manifold $M$ of singularity order $2,$ i.e. with a stratification $s(M):=(s_0(M),s_1(M),s_2(M))$ for smooth submanifolds $s_j(M),j=0,1,2,$ where $s_2(M)$ is the second order edge of $M,$ while $N:=M\setminus s_2(M)$ is a manifold with first order edge. The latter means that we have a stratification $s(N)=(s_0(N),s_1(N))$ where $s_1(N)$ is the edge of $N$, moreover, $s_0(N)=N\setminus s_1(N)$ is smooth, and $s_1(N)$ has a neighbourhood $V$ in $N$ that has the structure of an $X^\Delta $-bundle over $s_1(N)$ for a smooth compact manifold $X$ where  $X^\Delta :=(\overline{\R}_+\times X)/(\{0\}\times X)$ is the cone with base $X.$ Analogously the corner structure of $M$ near $s_2(M)$ means that $s_2(M)$ has a neighbourhood $W$ in $M$ with the structure of a $B^\Delta$-bundle over $s_2(M)$ for a compact manifold $B$ with edge. One of the major issues of such a program is to establish algebras of operators $A$ over $M$ with a symbolic structure $\sigma(A):= (\sigma_0(A), \sigma_1(A), \sigma_2(A)) $ referring to the stratification of $M$ and to understand, in particular, the natural operations such as compositions, similarly as is known in the pseudo-differential analysis on a manifold with smooth edge, cf. \cite{Schu32}. Recall that in the smooth edge case $N$ the interior symbol $\sigma _0(A),$ defined on $s_0(N),$ is edge-degenerate while the principal edge symbol $\sigma _1(A)$ is defined on
$s_1(N)$ and takes values in operators on the infinite stretched cone belonging to the above-mentioned cone bundle. The  edge calculus has much in common with boundary value problems on a smooth manifold with boundary, where the boundary is just the edge and $\sigma _1(A)$ is the well-known boundary symbol, see, in particular, Eskin's book \cite{Eski2} or Boutet de Monvel's algebra \cite{Bout1}. The point of view of boundary symbols turned out to be very fruitful in the subsequent development, see also \cite{Remp2}, or Grubb's monograph \cite{Grub1}. A calculus with edge-degenerate pseudo-differential operators and an extension of such a symbolic structure has been constructed in \cite{Schu32}, partly based of the investigation of boundary value problems without the transmission property \cite{Remp1}. In particular, the Mellin quantisations of \cite{Eski2} have influenced the way of building up the cone and edge theories that may be found now in various monographs, such as \cite{Schu2}, \cite{Egor1}. The paper \cite{Schu27} developed a corner calculus for the case $\textup{dim}\,s_2(M)=0$ which is to some extent a generalisation of the cone algebra to the case of a singular base of the cone. Also the pseudo-differential edge analysis has been further developed, especially through an alternative edge quatisation in \cite{Gil2}, compared with the one employed before, here referred to as a Mellin-edge quantisation. Note that, although this belongs to the symbolic structure of edge pseudo-differential operators, it concerns, in fact, the analysis of specific parameter dependent families of operators on an infinite stretched cone. Let us also point out that here we refer to algebras of singular operators with a control of asymptotics of solutions to elliptic equations close to the singularities. More material and applications may be found in \cite{Dine4}, \cite{Liu3}, \cite{Haru13}, \cite{Flad3}, \cite{Kapa15}, \cite{Schu56}, \cite{Schu41}, \cite{Hab1}, \cite{Schu74}.

The main intention of the present article is to pass to the next singularity order and to construct via a Mellin-edge quantisation the corresponding parameter dependent families of operators on an infinite stretched cone, now with a base $B$ that has edge singularities of first order. An inspection of \cite{Gil2} shows the usefulness of such quantisations for managing the edge calculus. For similar reasons it is desirable to do such things for the corner calculus. However, since the infinite stretched cone with base $B$  has an edge that is tending to infinity together with the corner axis variable, and because of the corner of $B^\Delta $ itself, some difficulties appear in the second order corner case for the first time. We show here that such a quantisation is possible, indeed. Formally some consequences are then similar to the edge case, and the applications to ellipticity and parametrix constructions of corner operators within the calculus will be given in a subsequent paper.

The exposition is organised as follows. In Section 1 we formulate the spaces of parameter-dependent operators on a manifold with smooth edge, cf. Definition \ref{1-95}. Those constitute the parameter-dependent edge algebra. For the case of a smooth edge we briefly recall what we understand by edge and Mellin-edge quantisation. The parameter-dependent edge calculus is basic for similar quatisations on the level of edges of second singularity order. Moreover, we give some new descriptions of weighted spaces of smoothness $s\in \N$ over the infinite cone, cf. Proposition \ref{1-79}, and over a wedge, cf. Proposition \ref{1-84}, in terms of (for the calculus) typical differential operators. Those play a role for analogous descriptions of the respective weighted corner spaces below. Section 2 is devoted to the iteratively defined corner Sobolev spaces with double weights. After the definition for arbitrary smoothness, cf. Definitions \ref{2-8}, \ref{1-sp}, we characterise the corresponding spaces for the case $s\in \N$ in terms of differentiations, cf. Propositions \ref{2-17}, \ref{2-18}. The equivalence of edge and Mellin-edge quantisation in the corner case is established in Section 3.  In Section 4 we see that there remain corner Green symbols with flatness of infinite order in the corner-axis variable $t$ for $t\rightarrow 0$ as well as in the inner cone-axis variable $r$ for $r\rightarrow 0.$ Here we also employ a norm growth result for parameter-dependent edge operators that has been prepared in Section 1, cf. Theorem \ref{1-8}. Finally in Section 5 we show a growth comparing Theorem \ref{1-18} which is crucial for the proof of Theorem \ref{mainth}.

\section{Parameter-dependent operators on a manifold with edge}
Parameter-dependent operators in a similar sense as in \cite{Agra1} or \cite{Kond1} play a role in the analysis of operators on a singular manifold as ``semi-quantised" objects or operator-valued amplitude functions in pseudo-differential operators. Such families of operators appear on different levels, parallel to the stratification of the corresponding underlying space. In this article we consider stratified spaces $M=\bigcup _{j=0}^k s_j(M)$ for $k=2,$ where the stata $s_j(M)$ are smooth manifolds of different dimensions, and $s_j(M)\cap s_l(M)=\emptyset$ for $j\neq l.$ Simplest examples for $k=1$ are manifolds with conical singularities or edges. In order to illustrate some elements of our approach in general we assume for the moment $k$ to be arbitrary, cf. also \cite{Schu57}. A guideline to build up an ``adequate" pseudo-differential algebra over $M$ may be to first single out an algebra of (for the singularity) typical differential operators $A,$ to recognise the principal symbolic structure $\sigma (A):=(\sigma _j(A))_{0\leq j\leq k},$ in fact, a principal symbolic hierarchy, and then to ask a pseudo-differential algebra containing the elliptic elements (with respect to the $\sigma (\cdot)$) together with the parametrices of elliptic elements. Besides $\sigma (A)$ there are also required local amplitude functions to determine the operators $A$ via some chosen integral transform, here the Fourier or the Mellin transform. In our approach the raw material are scalar symbols of rather standard classes, however, degenerate in local splittings of variables on so-called stretched manifolds. Those already determine the operators modulo smoothing ones over $ s_0(M)$ (the main stratum), when we simply apply the operator convention (``quantisation") based on the local Fourier transform. Depending on the expectation what the future operator calculus has to accomplish  continuity of operators in weighted Sobolev spaces, subspaces with asymptotics of different kind, or elliptic regularity to solutions in such spaces, the operator conventions have to be reorganised, according to the nature of the singularities, on the expense of certain smoothing operators over $ s_0(M).$ Since the resulting new amplitude functions a priori contain the expected information, their construction may be the main task to understand the above-mentioned regularity phenomena.

Let $L_\textup{cl}^\mu(X;\R^l)$ for a $C^\infty$ manifold $X$ (with a fixed
Riemannian metric) be the space of all parameter-dependent
 pseudo-differential operators on $X$ of order $\mu\in\R$, where
 $\lambda \in \R^l$, $l\in\N$, is the parameter (for $l=0$ we omit
 $\R^l$). Modulo $L^{-\infty}(X;\R^l)$ an operator $A\in L_{\textup{cl}}^\mu(X;\R^l)$
 is locally given in the form $\textup{Op}_x(a)(\lambda)$ for an
 amplitude function $a(x,\xi,\lambda)$ of order $\mu$ over
 $\R_x^n$ of H\"omander's class with $(\xi,\lambda)\in \R^n\times
 \R^l$ being the covariable (classical, when $a$ has an asymptotic
 expansion into homogenous components of order $\mu-j$, $j\in \N$). Parameter-dependent pseudo-differential operators are very common; our definition is slightly stonger that the one in Shubin's book \cite{Shub2}. By $\textup{Op}_x(\cdot)$ we
 understand the operator based on the Fourier transform in $\R^n$,
 i.e. $\mathcal{F}^{-1}a \mathcal{F}$; moreover
 $L^{-\infty}(X;\R^l):=\mathcal{S}(\R^l,L^{-\infty}(X))$,
 where $L^{-\infty}(X)$ is the space of smoothing operators on
 $X$. Let $H^s(X)$, $s\in\R$, denote the scale of standard Sobolev
 spaces on $X$ (when $X$ is compact, otherwise we have
 $H_{\textup{loc}}^s(X)$, the space of distributions $u$ on $X$
 such that for every $\varphi\in C_0^\infty(X)$ the push forward
 of $\varphi u$ under a chart to $\R^n$ belongs to $H^s(\R^n)$, and $H^s_{\textup{comp}}(X)$ the subspace of elements
 with compact support). We often employ the well-known fact that
 for a compact $C^\infty$ manifold $X$ the space
 $L^\mu_{\textup{cl}}(X;\R^l)$ for every $\mu\in \R$ contains an
 element $R^\mu(\lambda)$ that induces isomorphisms
 \begin{equation}\label{1-1}R^\mu(\lambda):H^s(X)\to
 H^{s-\mu}(X) \quad \textup{for all}\, s\in \R.\end{equation}
 For references below we need some notation
 from the calculus on a manifold $M$ with conical singularities.
 Assume for simplicity that there is only one singular point $c$.
 Locally close to $c$ we express spaces and operators in the variables $(r,x)\in
 X^\wedge.$ Weighted distribution spaces will be described in terms of the Mellin transform in $r$-direction, $Mu(z)=\int_0^\infty r^{z-1}u(r)dr.$ Generalities on Mellin transformation techniques may also be found in \cite{Dors1}. First, $\mathcal{H}^{s,\gamma}(X^\wedge)$ for $s,\gamma\in\R$ is the completion of $C_0^\infty(X^\wedge)$ with respect
 to the norm \begin{equation}\label{1-34}\big\{\int\parallel  \!R^s(\textup{Im}\,z)Mu(z)\!\parallel  ^2_{L^2(X)}\dbar z\big\}^\frac{1}{2},\end{equation}
 $\dbar z:=(2\pi i)^{-1}dz$, with integration
 over the line $\Gamma_{\frac{n+1}{2}-\gamma}$, $n:=\textup{dim}\,X$, where
 $\Gamma_\beta:=\{z\in\C:\textup{Re}\, z=\beta\}$. Here $R^s(\lambda)\in
 L_{\textup{cl}}^s(X;\R)$ is an order reducing family of the kind
 \eqref{1-1} for $\mu=s$, and $l=1$.
 \\
 \\
 Observe that
 \begin{equation}\label{0-6}
 \mathcal{H}^{0,0}(X^\wedge)=r^{-n/2}L^2(\R_+\times X)
 \end{equation} with $L^2(\R_+\times X)$ referring to the measure
 $drdx$. Spaces like $\mathcal{H}^{s,\gamma}(X^\wedge)$ already occur in \cite{Kond1}.
 It will also be useful to
 have the cylindrical Sobolev spaces $H^s(\R^q\times X)$ defined
 as the completion of $C_0^\infty(\R^q\times X)$ in the norm
$\big\{\int \parallel  \!R^s(\eta)\hat{u}(\eta)\!\parallel  ^2_{L^2(X)}d\eta\big\}^\frac{1}{2}$
 with $\hat{u}=\mathcal{F}_{y\to\eta} u$ being the Fourier transform
 and $R^s(\eta)$ an order reducing family in the above sense, here
 with the parameters $\eta\in \R^q$. Finally, we need spaces in
 the set-up of operators on the manifold $X^\asymp:=\R\times X$
 with conical exit to infinity. The space
 $H^s_{\textup{cone}}(X^\asymp)$ is defined to be the completion
 of $C_0^\infty(X^\asymp)$ with respect to the norm
\begin{equation}\label{1-35}\big\{\int_{-\infty}^\infty \parallel [r]^{-s+\frac{n}{2}}
 \mathcal{F}^{-1}_{\rho\to r}R^s([r]\rho,[r]\eta)(\mathcal{F}_{r\to \rho} u)(r)\parallel^2_{L^2(X)}dr\big\}^\frac{1}{2}\end{equation}
 where $R^s(\tilde{\rho},\tilde{\eta})$ is order reducing on $X$
 with the parameters $(\tilde{\rho},\tilde{\eta})\in \R^{1+q}$,
 $\mathcal{F}_{r\to\rho}$ the one-dimensional Fourier transform,
 and $r\to [r]$ is a function in $C^\infty(\R)$, $[r]>0$, with $[r]=|r|$
 for $|r|\geq c$ for some $c>0$. In the latter norm we choose $\eta\in
 \R^q$ fixed and $|\eta|$ sufficiently large. Moreover, we set
 $H^s_{\textup{cone}}(X^\wedge):=H^s_{\textup{cone}}(X^\asymp)|_{X^\wedge}$.
 In this paper a cut-off function on the half-axis will be any real valued $\omega\in C^\infty(\overline{\R}_+)$
 such that $\omega(r)=1$ in a neighbourhood of $r=0$.
 For purposes below we recall another equivalent definition of the space $H^s_{\textup{cone}}(X^\asymp)$.
 To this end we choose a diffeomorphism $\chi_1: U\to U_1$ from a coordinate neighbourhood $U$ on $X$ to an
 open set $U_1\subset S^n$ with $S^n$ being the unit sphere in $\R_{\tilde{x}}^{n+1}$. Moreover, we form
 the diffeomorphism $\chi:\R_+\times U\to \Gamma:=\{\tilde{x}\in \R^{n+1}\setminus \{0\}: \tilde{x}/|\tilde{x}|\in U_1\}$
 by setting $\chi(r,x):=r\chi_1(x)$. Then $H^s_{\textup{cone}}(X^\asymp)$ is the set of all $u\in H^s_{\textup{loc}}(\R\times X)$
 such that $(1-\omega(r))\varphi(x)(u(\pm r,x)|_{\R_+\times U})\in \chi^*H^s(\R^{n+1})|_{\Gamma}.$ for any such $\chi$ and for every $\varphi\in C_0^\infty(U).$
Instead of polar coordinates in $\Gamma\subset \R_{\tilde{x}}^{n+1}\setminus\{0\}$ it is sufficient
 to assume that $U$ is diffeomorphic to the open unit ball in $\R^n$ and to write the coordinates
 $\tilde{x}\in \Gamma $ in the form $\tilde{x}={\tilde{x}_0\tilde{x}^\prime}$ for $\tilde{x}_0\in \R$, $\tilde{x}^\prime\in \R^n$,
 and to take a specific position of $\Gamma$ in the half space $\tilde{x}_0>0$ where we identify $\tilde{x}_0$ with
 $r\in \R_+$. We write $\Gamma$ in the form
 \[\Gamma=\{(r,r\tilde{x}^\prime)\in \R^{n+1}: r\in \R_+, \tilde{x}^\prime\in \R^n, |\tilde{x}^\prime|<1\}.\]
 Then if we identify $U$ with $|x|<1$ with $x\in \R^n$ being local coordinates in $U$, the diffeomorphism $\chi_1$ may be taken
 as an identification of $x$ with $\tilde{x}^\prime$ for $|\tilde{x}^\prime|<1$ and the above $\chi$ takes the form
 $\chi(r,x)=(r,rx)$. Then the space $H^s_{\textup{cone}}(X^\asymp)$ can be characterised as the set of all $u\in H^s_{\textup{loc}}(\R\times X)$
such that for $u(\pm r,x)$ we have
 \[(1-\omega(r))\varphi(x)(u(\pm r,x)|_{\R_+\times U})=\tilde{u}(r,\frac{\tilde{x}^\prime}{r})\]
 for some $\tilde{u}(r,\tilde{x}^\prime)\in H^s(\R^{1+n})|_\Gamma$. In other words, defining
 \begin{equation}\label{1-85}
 \beta:\R_+\times U\to \Gamma
 \end{equation} by $\beta(r,x)=(r,rx)\subset \R^{n+1}$, the map $(\beta_{[r]}\tilde{u})(\tilde{x}):=\tilde{u}(r,[r]x)=u(r,x)$
 gives us a bijection
 \begin{equation}\label{1-86}
 \beta_{[r]}:(1-\omega(|\tilde{x}|))\tilde{\varphi}H^s(\R^{n+1})|_\Gamma\to (1-\omega(r))\varphi(x)H^s_{\textup{cone}}(U^\wedge)
 \end{equation}
 for any $\varphi (x)\in C_0^\infty (U), \tilde{\varphi }(\tilde{x})\in C^\infty (\Gamma )$, such that $\tilde{\varphi }(r,rx)=\tilde{\varphi }(1,x)=\varphi (x)$ for all $r>0$ and a cut-off function $\omega (r)$ that is equal to $\omega ([r])$ where 
 $H^s_{\textup{cone}}(U^\wedge)=H^s_{\textup{cone}}(X^\wedge)|_{U^\wedge}.$

 A crucial
 role for the edge calculus play the spaces
 $$\mathcal{K}^{s,\gamma}(X^\wedge):=\{\omega u_0+(1-\omega)u_\infty: u_0\in
 \mathcal{H}^{s,\gamma}(X^\wedge), u_\infty\in H^s_{\textup{cone}}(X^\wedge)\}$$
 for any cut-off function $\omega $.
 Similarly as \eqref{0-6} we have $\mathcal{K}^{0,0}(X^\wedge)=r^{-n/2}L^2(\R_+\times X)$.
 Moreover, for every $s,\gamma, e\in \R$ we form the spaces
 $$\mathcal{K}^{s,\gamma;e}(X^\wedge):=\langle r\rangle^{-e}\mathcal{K}^{s,\gamma}(X^\wedge).$$ For purposes below we introduce a
 family of isomorphisms $\kappa_\delta:\mathcal{K}^{s,\gamma}(X^\wedge)\to \mathcal{K}^{s,\gamma}(X^\wedge),\quad
 (\kappa_\delta u)(r,x):=\delta^{\frac{n+1}{2}}u(\delta r,x)$, where $n:=\textup{dim}\,X$. Below we will also employ
 subspaces $\mathcal{K}^{s,\gamma}_{\Theta}(X^\wedge)\subset \mathcal{K}^{s,\gamma}(X^\wedge)$ of flatness $-\vartheta$
 relative to $\gamma$ for some $-\infty\leq \vartheta<0$ and the half-open weight interval $\Theta=(\vartheta,0].$ More precisely,
  we set $\mathcal{K}^{s,\gamma}_{\Theta}(X^\wedge)=\lim_{_{\begin{subarray}{c}\longleftarrow\\ \varepsilon>0\end{subarray}}}
  \mathcal{K}^{s,\gamma-\vartheta-\varepsilon }(X^\wedge)$
  in the respective Fr\'echet topology. In the case $\Theta=(-\infty,0]$ the space
  $\mathcal{K}^{s,\gamma}_\Theta(X^\wedge)=\mathcal{K}^{s,\infty}(X^\wedge)$ is independent of $\gamma$.
  \\
  \\
  Below we will need more information about the spaces $\mathcal{H}^{s,\gamma}(X^\wedge)$, $H^s_{\textup{cone}}(X^\wedge)$, etc.
  Let us set $(S_{\gamma-n/2}u)(\boldsymbol{r},x)=e^{-((n+1)/2-\gamma)\boldsymbol{r}}u(e^{-\boldsymbol{r}},x),$ for
  $(\boldsymbol{r},x)\in \R\times X$. This transformation extends $ S_{\gamma-n/2}: C_0^\infty(X^\wedge)\to C^\infty_0(\R\times X)$
  to an isomorphism
  \begin{equation}\label{1-77}
  S_{\gamma-n/2}:\mathcal{H}^{s,\gamma}(X^\wedge)\to H^s(\R\times X)
  \end{equation} for every $s,\gamma\in \R$.
  Another useful property is the following characterisation. For future references by $\textup{Diff}^m(\cdot)$ we denote the space of differential operators of order $m$ on the smooth manifold in parentheses.
  \begin{Prop}\label{1-91}
  For $s\in \N$, $\gamma\in \R$, the condition $u(r,x)\in \mathcal{H}^{s,\gamma}(X^\wedge)$ is equivalent to
  $(r\partial_r)^jD_x^\alpha u(r,x)\in r^{-n/2+\gamma}L^2(\R_+\times X)$ for all $D_x^\alpha \in \textup{Diff}^\alpha(X)$
  and $j+|\alpha|\leq s$.
  \end{Prop}
Proposition \ref{1-91} is well-known. Recall that he proof easily follows by an analogous characterisation of the cylindrical Sobolov spaces $H^s(\R\times X)$ in terms of
differential operators in $(\boldsymbol{r},x)$ combined with \eqref{1-77}.
\begin{Prop}\label{1-79}
For every $s\in \N$ there exists a finite system of operators $\{D_j^\alpha: D_j^\alpha\in
 \textup{Diff}^{|\alpha|}(X^\wedge), |\alpha|\leq s, j=1,...,N\}$, for a suitable $N$,
such that $u\in \mathcal{K}^{s,\gamma}(X^\wedge)$ is equivalent to
\begin{equation}\label{1-78}
u\in \mathcal{K}^{0,\gamma}(X^\wedge), D_j^\alpha u\in \mathcal{K}^{0,\gamma;-s}(X^\wedge), \quad \textup{for all}\quad
|\alpha|\leq s, j=1,...,N
\end{equation} for every $\gamma\in \R$, and we have
\begin{equation}\label{1-81}
\| u\|_{\mathcal{K}^{s,\gamma}(X^\wedge)}\sim\{\|u\|^2_{\mathcal{K}^{0,\gamma}(X^\wedge)}+\sum_{0<|\alpha|\leq s}
\sum_{j=1,...,N}\|D_j^\alpha u\|^2_{\mathcal{K}^{0,\gamma;-|\alpha|}(X^\wedge)}\}^{1/2}.
\end{equation} Moreover, the norm $\{\|u\|^2_{\mathcal{K}^{0,\gamma}(X^\wedge)}+\sum_{|\alpha|=s}
\sum_{j=1,...,N}\|D_j^\alpha u\|^2_{\mathcal{K}^{0,\gamma;-s}(X^\wedge)}\}^{1/2}$ is equivalent to \eqref{1-81}.
The operators $D_j^\alpha$ can be chosen in such a way that
\begin{equation}\label{1-82}
\delta_\lambda D_j^\alpha=D_j^\alpha\delta_\lambda, \quad \textup{for all}\,\, \lambda\in \R_+
 \end{equation} for every $j, \alpha$, where $(\delta_\lambda u)(r,x):=u(\lambda r,x)$, $\lambda\in \R_+$.
\end{Prop}
\begin{proof}
Let $\{U_1,...,U_N\}$ be an open covering of $X$ by coordinate neighbourhoods,
$\{\phi_1,...,\phi_N\}$ a subordinate partition of unity, and fix diffeomorphisms
$\chi_{j,1}: U_j\to V_j$ for suitable open subsets $V_j\subset S^n$. Then we obtain
diffeomorphisms $\chi_j : \R_+\times U_j \to \Gamma_j$ for conical subsets $\Gamma_j\subset \R_{\tilde{x}}^{1+n}$,
$\Gamma_j:=\{\tilde{x}\in \R^{1+n}\setminus \{0\}: \tilde{x}/|\tilde{x}|\in V_j\}$, by
$\chi_j(r,x):=r\chi_{j,1}(x)$, $r>0$. Then the operators
$D^\alpha_j: u \mapsto r^{|\alpha|}D^\alpha_{\tilde{x}}((\Phi_j u)\circ \chi_j^{-1})$
are obviously as desired.
\end{proof} Observe that $r^{|\alpha|}D^\alpha_{\tilde{x}}$ can be expressed by operators of the form
$(r\partial_r)^jD_x^\beta$ for $|\alpha|=j+|\beta|$.
  Let $B$ be a compact
 manifold with edge, $Y:=s_1(B)$. By definition $B$ is locally near $Y$ modelled on a wedge $X^\Delta\times \Omega$, for a
 smooth compact manifold $X$ and an open set $\Omega\subseteq \R^q$, $q=\textup{dim}Y$, every $y\in Y$
 has a neighbourhood $V$ in $B$ such that there is a ``singular" chart
 \begin{equation}\label{1-26}
 \chi: V\to X^\Delta\times \Omega
 \end{equation} which restricts to a chart $\chi_{V\cap Y} : V\cap Y\to \Omega$ on $Y$ and to a diffeomorphism
 \begin{equation}\label{1-30}\chi_\textup{reg}: V\setminus Y \to X^\wedge\times \Omega\end{equation} to the corresponding open stretched wedge.
 For simplicity we assume the transition maps for the latter charts to be independent of $r$ for some $0<r<\varepsilon$, $\varepsilon >0$.
If $B$ is a manifold with edge $Y$ by $\B$ we denote its stretched manifold, a smooth manifold with boundary $\partial \B$, where
$\partial \B$ is an $X$ bundle over $Y.$ Then we have $B=\B/\sim$ under the equivalence relation that identifies points in $\partial\B$ with
the same projection $y\in Y$.

If $H$ is a Hilbert space and $\varkappa=\{\varkappa_\delta\}_{\delta\in \R_+}$
a group of isomorphisms $\varkappa_\delta: H\to H$ with $\varkappa_\delta \varkappa_{\delta^\prime}=\varkappa_{\delta\delta^\prime}$
which is strongly continuous in $\delta\in \R_+$ we call $\varkappa$ a group action on $H$. We then always have an estimate
\begin{equation}\label{normk}
\|\varkappa_\lambda\|_{\mathcal{L}(H)}\leq c\max\{\lambda, \lambda^{-1}\}^M
\end{equation}
for some constants $c,M>0.$
A similar notion makes sense if we replace
 $H$ by a Fr\'echet space $E$ written as a projective limit $E=\lim_{_{\begin{subarray}{c}\longleftarrow\\ j\in \N\end{subarray}}}E^j$
 of Hilbert spaces with continuous embeddings $E^j \hookrightarrow E^0$ for all $j$ where $\varkappa$ is a group action of $E^0$
  that restricts to a group action on $E^j$ for every $j$.
\begin{Def}\label{1-25}
\begin{itemize}\item[\textup{(i)}]
We define $\mathcal{W}^{s}(\R^q,H)$ as the completion of $\mathcal{S}(\R^q,H)$
with respect to the norm
$\Vert u\Vert_{\mathcal{W}^s(\R^q,H)}=\big\{\int\langle\eta\rangle^{2s}
\Vert \kappa^{-1}_{\langle\eta\rangle} \hat{u}(\eta)\Vert^2_{H}d\eta\big\}^{\frac{1}{2}}$
with $\hat{u}(\eta)=(\mathcal{F}_{y\to \eta}u)(\eta)$ being the Fourier transform in $\R^q$. In particular, we can set
$H:=\mathcal{K}^{s,\gamma}(X^\wedge)$ with the above mentioned group action. \item[\textup{(ii)}]The space $H^{s,\gamma}(B)$ is
defined to be the set of all $u\in H_\textup{loc}^s(B\setminus Y)$ such that for every singular chart $\chi: V\to X^\Delta\times \R^q$
on $B$, $\omega\psi u\circ \chi_{\textup{reg}}^{-1}\in \mathcal{W}^{s}(\R^q,\mathcal{K}^{s,\gamma}(X^\wedge))$ for any cut-off
function $\omega$ and $\psi\in C^\infty_0(\B)$ supported close to $\partial \B$ in the stretched version of $V$.\end{itemize}
\end{Def}
In the case of the trivial group action, i.e. $\varkappa_\lambda=\textup{id}$ for all
$\lambda\in \R_+$ we write
\begin{equation}\label{eins}
\mathcal{W}^s(\R^q,H)=\mathcal{W}_1^s(\R^q,H)
\end{equation}
which is the standard Sobolev space over $\R^q$ of $H$-valued distributions of
smoothness $s$. In particular, for $s=0$ it follows that $\mathcal{W}^0_1(\R^q,H)=L^2(\R^q,H).$
\begin{Lem}\label{1-92}
Let $\varkappa$ be the trivial group action in the space $H$. Then we have continuous embeddings
$\mathcal{W}_1^{s+M}(\R^q,H)\hookrightarrow \mathcal{W}^s(\R^q,H)\hookrightarrow\mathcal{W}_1^{s-M}(\R^q,H),$ and
$\mathcal{W}^{s+M}(\R^q,H)\hookrightarrow \mathcal{W}_1^s(\R^q,H)\hookrightarrow\mathcal{W}^{s-M}(\R^q,H) $
for all $s\in \R,$ with the constant $M$ in the estimate \eqref{normk}.
\end{Lem} In particular, for $\mathcal{W}^{\infty}=\bigcap_{s\in \R}\mathcal{W}^s$, etc., it follows that
$\mathcal{W}^{\infty}(\R^q,H)=\mathcal{W}_1^\infty(\R^q, H)$. Lemma \ref{1-92} is an evident consequence of Definition \ref{1-25} (i).
\begin{Prop}
For every $s\in \N$ we have $\mathcal{W}^s(\R^q,H)=\{u\in \mathcal{W}^0(\R^q,H): D_y^\alpha u\in \mathcal{W}^0(\R^q,H), |\alpha|\leq s\}$
and an equivalence of norms $$\| u\|_{\mathcal{W}^s(\R^q,H)}\sim \{\| u\|^2_{\mathcal{W}^0(\R^q,H)}+\sum_{|\alpha|\leq s}
\|D_y^\alpha u\|^2_{\mathcal{W}^0(\R^q,H)}\}^{1/2}.$$ The latter in turn is equivalent to $\{\| u\|^2_{\mathcal{W}^0(\R^q,H)}+\sum_{|\alpha|= s}
\|D_y^\alpha u\|^2_{\mathcal{W}^0(\R^q,H)}\}^{1/2}$.
\end{Prop}
\begin{proof}
 Let us consider the second norm expression. The first one may be discussed in a similar manner. The estimate
$c_1(1+\sum_{|\alpha|=s}|\eta^\alpha|^2)\!\!\!\leq \!\!\!\langle \eta\rangle^{2s}\!\!\!\leq \!\!\!c_2(1+\sum_{|\alpha|=s}|\eta^\alpha|^2),\,c_1, c_2\!\!>0,$ entails
$c_1\int(1+\sum_{|\alpha|=s}|\eta^\alpha|^2)\|\varkappa_{\langle\eta\rangle}^{-1}\hat{u}(\eta)\|^2_H d\eta\leq \| u\|^2_{\mathcal{W}^s(\R^q,H)}
\leq c_2\int(1+\sum_{|\alpha|=s}|\eta^\alpha|^2)\|\varkappa_{\langle\eta\rangle}^{-1}\hat{u}(\eta)\|^2_H d\eta
.$ We obtain, using $\eta^\alpha\varkappa_{\langle\eta\rangle}^{-1}\hat{u}(\eta)=\varkappa_{\langle\eta\rangle}^{-1}\eta^\alpha\hat{u}(\eta)$,
the relation
$\int(1+\sum_{|\alpha|=s}|\eta^\alpha|^2)\|\varkappa_{\langle\eta\rangle}^{-1}\hat{u}(\eta)\|^2_H d\eta=\| u\|^2_{\mathcal{W}^0(\R^q,H)}+\int
\sum_{|\alpha|=s}\|\varkappa_{\langle\eta\rangle}^{-1}\eta^\alpha\hat{u}(\eta)\|^2_H d\eta.$
Then $\eta^\alpha \mathcal{F}(u)(\eta)=\mathcal{F}(D_y^\alpha u)(\eta)$ gives us the assertion. 
\end{proof}
Note that \begin{equation}\label{0-7}
\mathcal{W}^0(\R^q,\mathcal{K}^{0,0}(X^\wedge))=r^{-n/2}L^2(\R_+\times X\times \R^q)\,\,\,\,\mbox{for}\,\,\,\,n=\textup{dim}\,X
\end{equation} where $L^2(\R_+\times X\times \R^q)$ refers to the measure $drdxdy$.
\begin{Prop}\label{1-84}
For every $s\in \N$ we have the following equivalence of norms
\begin{multline}\label{1-83}
\| u\|_{\mathcal{W}^s(\R^q,\mathcal{K}^{s,\gamma}(X^\wedge))}\sim
\{\| u\|^2_{\mathcal{W}^0(\R^q,\mathcal{K}^{0,\gamma}(X^\wedge))}
+\sum_{|\beta|=s}\|D_y^\beta u\|^2_{\mathcal{W}^0(\R^q,\mathcal{K}^{0,\gamma}(X^\wedge))}
\\+\sum_{|\alpha|=s,1\leq j\leq N}\| D_j^\alpha u\|^2_{\mathcal{W}^0(\R^q,\mathcal{K}^{0,\gamma;-s}(X^\wedge))}+
\sum_{|\alpha|=s,1\leq j\leq N,|\beta|=s}\| D_j^\alpha D_y^\beta u\|^2_{\mathcal{W}^0(\R^q,\mathcal{K}^{0,\gamma;-s}(X^\wedge))}\}^{1/2}.
\end{multline} In the sum over $\alpha$ we may equivalently take $0<|\alpha|\leq s$ and $\|\cdot\|_{\mathcal{W}^0(\R^q,\mathcal{K}^{0,\gamma;-|\alpha|}(X^\wedge))}$,
and $0<\beta\leq s$ in the sum over $\beta$.
\end{Prop}
\begin{proof}
According to Proposition \ref{1-79} we have
\begin{align*}
&\|u\|^2_{\mathcal{W}^s(\R^q,\mathcal{K}^{s,\gamma}(X^\wedge))}\sim \|u\|^2_{\mathcal{W}^0(\R^q,\mathcal{K}^{s,\gamma}(X^\wedge))}
+\sum_{|\beta|=s}\|D_y^\beta u\|^2_{\mathcal{W}^0(\R^q,\mathcal{K}^{s,\gamma}(X^\wedge))}
\\&=\int\|\langle\eta\rangle^{-(n+1)/2}\delta_{\langle\eta\rangle}^{-1}\hat{u}(r,x,\eta)\|^2_{\mathcal{K}^{s,\gamma}(X^\wedge)}d\eta\qquad\qquad\qquad\qquad\qquad\qquad
\\&\qquad +\sum_{|\beta|=s}\int\|\langle\eta\rangle^{-(n+1)/2}\delta_{\langle\eta\rangle}^{-1}
\widehat{(D_y^\beta u)}(r,x,\eta)\|^2_{\mathcal{K}^{s,\gamma}(X^\wedge)}d\eta
\\&=\int\|\langle\eta\rangle^{-(n+1)/2}\delta_{\langle\eta\rangle}^{-1}\hat{u}(r,x,\eta)\|^2_{\mathcal{K}^{0,\gamma}(X^\wedge)}
\qquad\qquad\qquad\qquad\qquad\qquad\qquad\\&\qquad +\sum_{|\alpha|=s,j=1,...,N}
\|\langle\eta\rangle^{-(n+1)/2}D_j^\alpha\delta_{\langle\eta\rangle}^{-1}\hat{u}(r,x,\eta)\|^2_{\mathcal{K}^{0,\gamma;-s}(X^\wedge)}d\eta\\&\qquad +\sum_{|\beta|=s}\int
\|\langle\eta\rangle^{-(n+1)/2}\delta_{\langle\eta\rangle}^{-1}\widehat{(D_y^\beta u)}(r,x,\eta)\|^2_{\mathcal{K}^{0,\gamma}(X^\wedge)}d\eta\\&\qquad +\sum_{|\alpha|=s,j=1,...,N,|\beta|=s}\int
\|\langle\eta\rangle^{-(n+1)/2}\delta_{\langle\eta\rangle}^{-1}D_j^\alpha\widehat{(D_y^\beta u)}(r,x,\eta)\|^2_{\mathcal{K}^{0,\gamma;-s}(X^\wedge)}d\eta;
\end{align*} here we employed the relation \eqref{1-82}.
\end{proof}
In order to prepare the operators of the edge calculus we first consider edge-degenerate pseudo-differential operators.
Let $L^\mu_\textup{deg}(X^\Delta\times \Omega;\R^l)$ denote the set of all $A(\lambda)\in L^\mu_\textup{cl}(X^\wedge\times \Omega;\R^l)$ such that
there is a $\tilde{p}(r,y,\tilde{\rho},\tilde{\eta},\tilde{\lambda})\in C^\infty(\overline{\R}_+\times \Omega,L^\mu_{\textup{cl}}(X;\R^{1+q+l}))$
with $A(\lambda)=r^{-\mu}\textup{Op}_y(\textup{Op}_r(p))(\lambda)$ \textup{mod} $L^{-\infty}(X^\wedge\times\Omega;\R^l)$ where $p(r,y,\rho,\eta,\lambda):=\tilde{p}(r,y,r\rho,r\eta,r\lambda)$. The space $L^\mu_\textup{deg}(B;\R^l)$ of parameter-dependent edge-degenerate pseudo-differential operators of order $\mu$
on a manifold $B$ with edge $Y$ is defined to be the subspace of all $A(\lambda)\in L^\mu_{\textup{cl}}(B\setminus Y;\R^l)$
such that for every singular chart $\chi$ as in \eqref{1-26} we have $A(\lambda)|_V\in (\chi^{-1})_*L^\mu_\textup{deg}(X^\Delta\times \Omega;\R^l)$.
The operators in $L^\mu_{\textup{deg}}(B;\R^l)$ will not be continuous between weighted
 spaces as in Definition \ref{1-25} (ii). Therefore, we modify the operators by smoothing ones via some quantisations
 close to the edge. After that, in order to have the expected natural properties under algebraic operations, we add specific so-called smoothing Mellin plus Green terms. The quantisation refers to the local
 splitting of variables $(r,x,y)$ coming from \eqref{1-30}. If $E$ is a Fr\'echet space by $\mathcal{A}(\C,E)$ we denote the space of entire functions with values in $E.$ Now let $M_{\mathcal{O}}^\mu(X;\R^l)$, $l\in \N$,
 denote the space of all $h(z,\lambda)\in \mathcal{A}(\C,L^\mu_{\textup{cl}}(X;\R^l))$ such that $h|_{\Gamma_\beta}\in L^\mu_{\textup{cl}}
 (X;\Gamma_\beta\times\R^l)$ for every $\beta\in\R$, uniformly in compact $\beta$-intervals. In the following we often employ the weighted Mellin transform in $\R_+$-direction $M_\gamma  u(z)=\int_0^\infty r^{z-1}u(r)dr|_{\Gamma_{1/2-\gamma }}$ with the inverse $M_\gamma  ^{-1}g(r)=\int_{\Gamma_{1/2-\gamma }}r^{-z}g(z)\dbar z,\,\dbar z=(2\pi i)^{-1}dz.$ The associated weighted Mellin pseudo-differential operators with Mellin amplitude function $f(r,r',z)$ are denoted by $\textup{op}^\gamma _M(f).$ We employ the well-known fact, cf., for instance, \cite{Schu2}, that for every
 $p(r,y,\rho,\eta,\lambda)=\tilde{p}(r,y,\tilde{\rho},\tilde{\eta},\tilde{\lambda})$ such that $\tilde{p}(r,y,\tilde{\rho},\tilde{\eta},\tilde{\lambda})
 \in C^\infty(\overline{\R}_+\times \Omega, L^\mu_{\textup{cl}}(X;\R^{1+q+l}))$ there exists an $\tilde{h}(r,y,z,\eta,\lambda)=
 \tilde{h}(r,y,z,r{\eta},r{\lambda})$ for $\tilde{h}(r,y,z,\tilde{\eta},\tilde{\lambda})\in C^\infty(\overline{\R}_+\times \Omega, M^\mu_{\mathcal{O}}(X;\R^{q+l}_{\tilde{\eta},\tilde{\lambda}}))$ such that
 \begin{equation}\label{1-39}\textup{Op}_r(p)(r,y,\lambda)=\textup{op}_M^\gamma (h)(y,\eta,\lambda)\end{equation}
 mod $C^\infty(\Omega,L^{-\infty}(X^\wedge;\R^{q+l}))$ for every $\gamma \in \R$. This result is also referred to as the Mellin quantisation of $p$ (or $\tilde{p}$).

In the following, for functions
 $\varphi$, $\varphi^\prime$ we write $\varphi\prec \varphi^\prime$ if $\varphi^\prime \equiv 1$ on $\textup{supp}\,\varphi$. Fix cut-off functions
 $\omega(r)$, $\omega^\prime(r)$ on the half-axis such that $\omega\prec\omega^\prime$ and excision function $\chi\prec \chi^\prime$ where $\omega
 +\chi=1$. Moreover, we set $\varphi_\eta(r):=\varphi(r[\eta])$ when $\varphi$ is a
 function on the half-axis, where $\eta\to [\eta]$ is any smooth strictly positive function in $\R^q$ such that $[\eta]=|\eta|$ for large $|\eta|$.
 Then, by virtue of pseudo-locality we have
 \begin{equation}\label{1-31}
 r^{-\mu}\textup{Op}_r(p)(y,\eta,\lambda)=r^{-\mu}\{\omega_{\eta,\lambda}\textup{op}_M^{\gamma-\frac{n}{2}}(h)(y,\eta,\lambda)\omega^\prime_{\eta,\lambda}
 +\chi_{\eta,\lambda}\textup{Op}_r(p)(y,\eta,\lambda)\chi^\prime_{\eta,\lambda}\}
 \end{equation}
mod $C^\infty(\Omega,L^{-\infty}(X^\wedge;\R^{q+l}))$. For the structure of the edge calculus it is essential that the
right hand side of \eqref{1-31} gives rise to an operator-valued symbol
\begin{equation}\label{1-32}
a(y,\eta,\lambda):=r^{-\mu}\epsilon\{\omega_{\eta,\lambda}\textup{op}_M^{\gamma-\frac{n}{2}}(h)(y,\eta,\lambda)\omega^\prime_{\eta,\lambda}
+\chi_{\eta,\lambda}\textup{Op}_r(p)(y,\eta,\lambda)\chi^\prime_{\eta,\lambda}\}\epsilon^\prime
\end{equation} for any cut-off functions $\epsilon$, $\epsilon^\prime$ on the $r$-half-axis. The correspondence $p\rightarrow a$ for an edge-degenerate symbol $p$ and $a$ of the form \eqref{1-32} will be called an edge quantisation of $p.$\\
\\
 Let $U\subseteq \R^p$ be an open set and $H$ and $\tilde{H}$ Hilbert spaces
  with group action $\varkappa$ and $\tilde{\varkappa}$, respectively. Then $S^\mu(U\times \R^q; H,\tilde{H})$ is defined to be the
  set of all $a(y,\eta)\in C^\infty(U\times \R^q,\mathcal{L}(H,\tilde{H}))$ satisfying the following (twisted) symbolic estimates
  $$\|\tilde{\varkappa}_{\langle\eta\rangle}^{-1}\{D^\alpha_yD^\beta_\eta a(y,\eta)\}\varkappa_{\langle\eta\rangle}\|_{\mathcal{L}(H,\tilde{H})}
  \leq c\langle\eta\rangle^{\mu-|\beta|}$$ for all $(y,\eta)\in K\times \R^q$, $K\Subset U$, and all multi-indices $\alpha\in \N^p$,
  $\beta\in \N^q$, for constants $c=c(\alpha,\beta,K)>0$. If $a_{(\mu)}(y,\eta)\in C^\infty(U\times (\R^q\setminus\{0\}),\mathcal{L}(H,\tilde{H}))$
  is a function satisfying the (twisted) homogeneity relation
  \begin{equation}\label{1-33}
  a_{(\mu)}(y,\delta\eta)=\delta^\mu\tilde{\varkappa}_\delta a_{(\mu)}(y,\eta)\varkappa_\delta^{-1}
  \end{equation} for all $\delta\in \R_+$, then we have $\chi(\eta)a_{(\mu)}(y,\eta)\in S^\mu(U\times \R^q; H,\tilde{H})$ for any
  excision function $\chi$ in $\R^q$ (i.e. $\chi\in C^\infty(\R^q)$, $\chi(\eta)\equiv 0$ for $|\eta|\leq c_0$,
  $\chi(\eta)\equiv 1$ for $|\eta|\geq c_1$, for some $0<c_0<c_1$). An element $a\in S^\mu(U\times \R^q; H,\tilde{H})$ is
  called classical if there is a sequence of homogeneous components $a_{(\mu-j)}(y,\eta)$, $j\in \N$, satisfying analogous
  relations to \eqref{1-33} for $\mu-j$ rather than $\mu$, for all $j$, such that
  $a(y,\eta)-\chi(\eta)\sum^N_{j=0}a_{(\mu-j)}(y,\eta)\in S^{\mu-(N+1)}(U\times \R^q; H,\tilde{H})$ for
  every $N\in \N$ and any excision function $\chi$. Let $S^\mu_{\textup{cl}}(U\times \R^q; H,\tilde{H})$ denote the space of classical
  symbols. A similar notion makes sense when $H$ or $\tilde{H}$ are Fr\'echet spaces with group action. The following observation is well-known. 
\begin{Lem}\label{1-80}
Let $a(y,\eta)\in C^\infty(U\times \R^q,\mathcal{L}(H,\tilde{H}))$ be a function such that
$a(y,\lambda\eta)=\lambda^\mu\tilde{\varkappa}_\lambda a(y,\eta)\varkappa_\lambda ^{-1}$
for $\lambda\geq 1$, $|\eta|\geq c$, for a constant $c>0$. Then we have
$a(y,\eta)\in S^\mu_{\textup{cl}}(U\times\R^q;H,\tilde{H}).$
\end{Lem}  
This is a simple consequence of the twisted symbolic estimates. In the following definition we employ some terminology about analytic functionals in the complex plane. As in \cite{Horm5} by $\mathcal{A}'(K)$ for any compact set $K\subset \C$ we denote the space of analytic functionals carried by $K.$ This is a nuclear Fr\'echet space, cf. also \cite{Kapa10}. If $E$ is a Fr\'echet space, by $\mathcal{A}'(K,E)$ we denote the space of all $E$-valued analytic functionals carried by $K$ in the topology of the projective tensor product $\mathcal{A}'(K)\hat{\otimes}_\pi E. $
 \begin{Def}\label{1-37} Let $\boldsymbol{g}=(\gamma,\Theta)$ be weight data, $\gamma\in \R$, $\Theta=(\vartheta,0]$,
$-\infty\leq \vartheta<0$. For $\vartheta>-\infty$ we choose any compact set $K\subset \C$, $K\subset \{\textup{Re}\,z<\frac{n+1}{2}-\gamma\}$
and form the space $$\mathcal{E}_K:=\{\omega\langle\zeta,r^{-z}\rangle:\zeta\in\mathcal{A}^\prime(K,C^\infty(X))\}$$ for any fixed cut-off function $\omega(r)$.
Then we set $\mathcal{K}^{s,\gamma}_P(X^\wedge):=\mathcal{E}_K+\mathcal{K}^{s,\gamma}_\Theta(X^\wedge)$ in the Fr\'echet topology of the non-direct
sum \textup{(}where $\mathcal{E}_K$ is topologised via the isomorphism $\mathcal{E}_K\cong \mathcal{A}^\prime(K,C^\infty(X))$\textup{)}. We call $P$ the continuous asymptotic type associated with the set $K\cap\{\textup{Re}\,z>\frac{n+1}{2}-\gamma+\vartheta\}$.
\end{Def}
    An element $g(y,\eta)\in \bigcap_{s,e\in\R}S^\mu_{\textup{cl}}(U\times\R^q;\mathcal{K}^{s,\gamma;e}(X^\wedge),
  \mathcal{K}^{\infty,\gamma-\mu;\infty}(X^\wedge))$ is called a Green symbol of order $\mu$, associated with the weight data
  $\boldsymbol{g}=(\gamma,\gamma-\mu,\Theta)$ if
 \begin{equation}\label{g}g(y,\eta)\in S^\mu_{\textup{cl}}(U\times\R^q;\mathcal{K}^{s,\gamma;e}(X^\wedge),
  \mathcal{K}_P^{\infty,\gamma-\mu;\infty}(X^\wedge)),
\end{equation}
\begin{equation}\label{g*}g^*(y,\eta)\in S^\mu_{\textup{cl}}(U\times\R^q;\mathcal{K}^{s,-\gamma+\mu;e}(X^\wedge),
  \mathcal{K}_Q^{\infty,-\gamma;\infty}(X^\wedge))
\end{equation} for every $s,e\in \R$, for some ${g}$-dependent (continuous) asymptotic
  types $P$ and $Q$, associated with the weight data $(\gamma-\mu,\Theta)$ and $(-\gamma,\Theta)$, respectively.
  Let $R^\mu_G(U\times \R^q,\boldsymbol{g})_{P,Q}$ denote the set of all Green symbols of order $\mu$, with asymptotic types $P$, $Q$.
 In the case $\Theta=(-\infty,0]$ and trivial $P$, $Q$ (i.e. with subscript $\Theta$ instead of $P$, $Q$) the respective space of Green
 symbols is independent of the weight data $\boldsymbol{g}$ and denoted by $R^\mu_G(U\times \R^q)_{\mathcal{O}}$. Recall that Green symbols are $(y,\eta )$-wise Green operators in the cone algebra over $X^\wedge.$ Kernel characterisations of such operators may be found in \cite{Seil2}. The terminlogy comes from boundary value problems in Boutet de Monvel's calculus and the role of Green's functions in solvability of elliptic problems, see, \cite{Bout1}.\\
 \\Compared with \eqref{1-32}
 there is an alternative quantisation, here called Mellin-edge quantisation, namely, $$b(y,\eta,\lambda):=r^{-\mu}\epsilon \textup{op}_M^{\gamma-\frac{n}{2}}(h)(y,\eta,\lambda)\epsilon^\prime,$$
 cf. \cite{Gil2}. We have $a(y,\eta,\lambda),b(y,\eta,\lambda)\in S^\mu(\Omega\times \R^{q+l};\mathcal{K}^{s,\gamma}(X^\wedge),\mathcal{K}^{s-\mu,\gamma-\mu}(X^\wedge))$
 for every $s\in \R,$ and 
\begin{equation}\label{eq} a(y,\eta,\lambda)-b(y,\eta,\lambda)\in R^\mu_G(\Omega\times \R^{q+l})_{\mathcal{O}}.\end{equation} 
  \\We also need another class of operator-valued symbols, the smoothing Mellin symbols. For convenience from now on we assume $\Theta=(-1,0]$; this allows us to observe asymptotic phenomena in the weight intervals
 $\{\frac{n+1}{2}-\gamma-1<\textup{Re}\,z <\frac{n+1}{2}-\gamma\}$
 and $\{\frac{n+1}{2}-(\gamma-\mu)-1<\textup{Re}\,z
 <\frac{n+1}{2}-(\gamma-\mu)\}$, and to minimise the technical effort with sequences of lower order
 Mellin symbols which are adequate when we take $\Theta=(-(k+1),0]$
 instead. Let $V\subset \C$ be a compact set and let $M^{-\infty}_V(X)$ be the space of all $f\in\mathcal{A}(\C\setminus V,L^{-\infty}(X))$
  such that $\chi_Vf|_{\Gamma_\beta}\in \mathcal{S}(\Gamma_\beta,L^{-\infty}(X))$ for every $\beta\in \R$, uniformly in compact $\beta$-intervals.
  The space $M^{-\infty}_V(X)$ is Fr\'echet in a natural way. Under the assumption $V\cap \Gamma_{\frac{n+1}{2}-\gamma}=\emptyset$ we can form
  that the weighted Mellin operator $\textup{op}_M^{\gamma-\frac{n}{2}}(f)(y)$ for any $f\in C^\infty(\Omega,M^{-\infty}_V(X))$. Then, if $\omega$,
  $\omega^\prime$ are arbitrary cut-off functions, the operator family \begin{equation}\label{1-36}m(y,\eta,\lambda):=r^{-\mu}\omega_{\eta,\lambda}\textup{op}_M^{\gamma-\frac{n}{2}}(f)(y)\omega^\prime_{\eta,\lambda}
   \end{equation}
represents an element of
  $S^\mu_{\textup{cl}}(\Omega\times \R^{q+l};\mathcal{K}^{s,\gamma}(X^\wedge),\mathcal{K}^{\infty,\gamma-\mu}(X^\wedge))$.
 Recall that when we change $\omega$ or $\omega^\prime$ we only change $m$ by a Green symbol. Let $R_{M+G}^\mu(\Omega\times \R^{q+l},\boldsymbol{g})$
 for $\boldsymbol{g}=(\gamma,\gamma-\mu,\Theta)$ denote the set of all operator functions $(m+g)(y,\eta,\lambda)$ for $m$ as in \eqref{1-36}
 and a Green symbol $g$ as in \eqref{g}, \eqref{g*} (for $\Omega$ instead of $U$ and $q+l$ rather than $q,$ and from now on for $\Theta=(-1,0]$). Moreover, let $R^\mu(\Omega\times
 \R^{q+l},\boldsymbol{g})$ be the set all $a+m+g$ for $m+g\in R^\mu_{M+G}(\Omega\times \R^{q+l},\boldsymbol{g})$ and $a$ of the form \eqref{1-32}.
 Let us now strengthen Definition \ref{1-25} (ii) by replacing $\mathcal{K}^{s,\gamma}(X^\wedge)$ by the Fr\'echet space $\mathcal{K}_P^{s,\gamma}(X^\wedge)$
  for some continuous asymptotic type $P$. It can easily be verified that there is a representation of $\mathcal{K}_P^{s,\gamma}(X^\wedge)$ as a
  projective limit of  Hilbert spaces
  $E^j\subset E^0:=\mathcal{K}^{s,\gamma}(X^\wedge)$ with group action which allows us to form the spaces $\mathcal{W}^s(\R^q,E^j),$ and then we set
$\mathcal{W}^s(\R^q,\mathcal{K}_P^{s,\gamma}(X^\wedge))=\lim_{_{\begin{subarray}{c}\longleftarrow\\ j\in \N\end{subarray}}}\mathcal{W}^s(\R^q,E^j).$

Moreover, similarly as Definition \ref{1-25} (ii) we form $H^{s,\gamma}_P(B)$ by replacing $\mathcal{K}^{s,\gamma}(X^\wedge)$ by $\mathcal{K}_P^{s,\gamma}(X^\wedge)$. Now the space $L^{-\infty}(B,\boldsymbol{g})$ for $\boldsymbol{g}=(\gamma,\gamma-\mu,\Theta)$ is defined
to be the subset of all $C\in L^{-\infty}(B\setminus Y)$ which induce continuous operators $C:H^{s,\gamma}(B)\to H^{\infty,\gamma-\mu}_P(B),\, C^*:H^{s,-\gamma+\mu}(B)\to H^{\infty,-\gamma}_Q(B)$ for all $s\in \R$ and some $C$-dependent asymptotic types $P$, $Q$. If $L^{-\infty}(B,\boldsymbol{g})
_{P,Q}$ denotes for the moment the space of all $C$ with fixed $P,Q$, we have a Fr\'echet space, and we set
\[L^{-\infty}(B,\boldsymbol{g};\R^l)_{P,Q}:=\mathcal{S}(\R^l,L^{-\infty}(B,\boldsymbol{g})_{P,Q}), \,\, \textup{and}\,\,
L^{-\infty}(B,\boldsymbol{g};\R^l):=\bigcup_{P,Q}L^{-\infty}(B,\boldsymbol{g};\R^l)_{P,Q}.\]
\begin{Def}\label{1-95} The space $L^\mu(B,\boldsymbol{g};\R^l)$ for $\boldsymbol{g}=(\gamma ,\gamma -\mu ,\Theta )$ is defined to be the subspace of all $A(\lambda)\in
L^\mu_\textup{deg}(B,\boldsymbol{g};\R^l)$ of the form $$A(\lambda):=A_{\textup{edge}}(\lambda)+A_{\textup{int}}(\lambda)
+C(\lambda)$$ where $C(\lambda)\in L^{-\infty}(B,\boldsymbol{g};\R^l)$, $A_{\textup{int}}(\lambda)\in L^\mu_{\textup{cl}}(B\setminus Y;\R^l)$ is arbitrary, vanishing in a neighbourhood of $Y$, and $A_{\textup{edge}}(\lambda)$ is locally near $Y$ in the splitting of variables $(r,x,y)\in \R_+\times X\times \Omega$, $\Omega\subseteq \R^q$ open, $q=\textup{dim}\,Y$, of the form $\textup{Op}_y(a)(\lambda)$ for an operator-valued amplitude function $a(y,\eta,\lambda)\in
R^\mu(\Omega\times \R^{q+l},\boldsymbol{g}).$
\end{Def}
\begin{Rem}
The space $L^\mu(B,\boldsymbol{g};\R^l)$ can be represented as a union of Fr\'echet subspaces. Those are essentially
labeled by the asymptotic types $P$, $Q$ in the involved Green symbols and the Mellin asymptotic types in the smoothing
Mellin symbols, cf. the formula \eqref{1-36}.
\end{Rem}

Let us define the principal symbolic structure of operators in $L^\mu(B,\boldsymbol{g};\R^l)$, consisting of pairs
$\sigma(A)=(\sigma_0(A),\sigma_1(A)).$ The first component is the parameter-dependent principal homogeneous symbol as a function on $C^\infty(T^*(B\setminus Y)\times \R^l\setminus 0)$ where
$0$ stands for $(\tilde{\xi},\lambda)=0$ with $\tilde{\xi}$ being the covariable in $T^*((B\setminus Y)\setminus 0)$. The second
component is defined as
\begin{align*}
&\sigma_1(A)(y,\eta,\lambda):=r^{-\mu}\{\omega(|\eta,\lambda|r)\textup{op}_M^{\gamma-n/2}(h_0)(y,\eta,\lambda)\omega^\prime(|\eta,\lambda|r)
\\&+\chi(|\eta,\lambda|r)\textup{Op}_r(p_0)(y,\eta,\lambda)\chi^\prime(|\eta,\lambda|r)\}\\&+r^{-\mu}\omega(|\eta,\lambda|r)
\textup{op}_M^{\gamma-n/2}(f)(y)\omega^\prime(|\eta,\lambda|r)+g_{(\mu)}(y,\eta,\lambda),
\end{align*}
where $h_0(r,y,z,\eta,\lambda):=\tilde{h}(0,y,z,r\eta,r\lambda)$, $p_0(r,y,\rho,\eta,\lambda):=\tilde{p}(0,y,r\rho,r\eta,r\lambda)$,
$f(y,z)$ as in \eqref{1-39}, and $g_{(\mu)}(y,\eta,\lambda)$ is the principal (twisted homogeneous) symbol
 of $g(y,\eta,\lambda)$ involved in $a(y,\eta,\lambda)\in R^\mu(\Omega\times \R^{q+l},\boldsymbol{g})$, cf. also Definition \ref{1-95}.
 For fixed $\boldsymbol{g}=(\gamma ,\gamma -\mu ,\Theta )$ we also employ spaces of edge operators of order $m\leq \mu $ for $\mu -m\in \N.$ The notation is as follows. First we set
 $L^{\mu-1}(B,\boldsymbol{g};\R^l):=\{A\in L^\mu(B,\boldsymbol{g};\R^l):\sigma (A)=0\}.$
 The elements of the latter space again have a pair of principal symbols, denoted for the moment by $\sigma ^{\mu -1}(A).$ Then vanishing of $\sigma ^{\mu -1}(A)$ determines a subspace $L^{\mu-2}(B,\boldsymbol{g};\R^l),$ and so on. In this way we inductively obtain the spaces $L^m(B,\boldsymbol{g};\R^l)$ for any $m\leq \mu ,\mu -m\in \N.$
 \begin{Rem}
 Observe that for any $ A\in L^m(B,\boldsymbol{g};\R^l),\,\alpha \in \N^l$, we have $D_\lambda ^\alpha A\in L^{m-|\alpha |}(B,\boldsymbol{g};\R^l).$
\end{Rem}
\begin{Thm}\label{1-8}
Let $A\in L^m(B,\boldsymbol{g};\R^l)$, $\boldsymbol{g}=(\gamma,\gamma-\mu,\Theta),\,\mu -m\in \N,$ be
a parameter-dependent edge operator, interpreted as a family of continuous operators
$$A(\lambda): H^{s,\gamma}(B)\to H^{s-\nu,\gamma-\mu}(B)$$ for some $\nu\leq m$.
Then we have $$\Vert A(\lambda)\Vert_{\mathcal{L}(H^{s,\gamma}(B),H^{s-\nu,\gamma-\mu}(B))}\leq c
\langle \lambda\rangle^{\textup{max}\{m,m-\nu\}+M}$$ for all $\lambda\in\R^l$ for some $c=c(\nu,s)>0$,
and a certain constant $M=M(\nu,s)>0$.
\end{Thm} The proof is given in \cite{Abed1}.
The following theorem compares growth properties of Fourier- and Mellin-based weighted spaces at infinity as well as at zero in terms of certain operator-valued symbols, cf. also Section 5 below for the corner case.
\begin{Thm}\label{1-66} Let $\omega$ be a cut-off function.
\begin{enumerate}
\item[\textup{(i)}] For every $s,\gamma,\gamma^\prime ,L\in \R$ there is a $g=g(L)$ such that
$(1-\omega_\eta)r^{-L}\in S^L_{\textup{cl}}(\R^q;\mathcal{H}^{s,\gamma}(X^\wedge),\mathcal{K}^{s,\gamma^\prime;g}(X^\wedge)).$
\item[\textup{(ii)}] For every $s,\gamma,\gamma^\prime,g\in \R$ there is an $L=L(g)$
such that $(1-\omega_\eta)r^{-L}\in S^L_\textup{cl}(\R^q;\mathcal{K}^{s,\gamma;g}(X^\wedge),\mathcal{H}^{s,\gamma^\prime}(X^\wedge)).$
\item[\textup{(iii)}]For every $s,\gamma',g^\prime, L \in \R$ we have $\omega_\eta r^L\in S^{-L}_{\textup{cl}}(\R^q; \mathcal{H}^{s,\gamma^\prime}(X^\wedge),\mathcal{K}^{s,\gamma^\prime+L;g^\prime}(X^\wedge)).$
\item[\textup{(iv)}]For every $s,\gamma,g\in \R$ and $L\geq 0$ we have $\omega_\eta r^L\in S^{-L}_{\textup{cl}}(\R^q; \mathcal{K}^{s,\gamma;g}(X^\wedge),\mathcal{H}^{s,\gamma}(X^\wedge)).$
\end{enumerate}
\end{Thm}
A proof is given in \cite[Section A3]{Gil2}.
\section{Weighted corner spaces}
Our next objective is to study weighted corner Sobolev spaces from the point of view of the iterative approach. The corner pseudo-differential calculus is an adequate choice of weighted
Sobolev spaces. Roughly speaking the former smooth manifold $X$ is
replaced by a compact manifold $B$ with edge $Y$ and we define our spaces over $\R_+\times B:=B^\wedge$.
The considerations refer to the representation of $B$ locally near $Y$ in the form $X^\Delta\times U$
for a coordinate neighbourhood $U$ of $Y$, diffeomorphic to $\R^q$, and we argue in terms of the stretched
variables $(r,x,y)\in X^\wedge\times \R^q$.  Recall that on the double $2\mathbb{B}$ of the stretched
manifold $B$ we have the spaces $\mathcal{H}^{s,\gamma_2}(2\mathbb{B})$. Let us interpret $1-\omega$ also
as a function on $2\mathbb{B}$ vanishing on the counterpart of $\mathbb{B}$ in the definition of the double.
We need an analogue of the edge spaces of Definition \ref{1-25}
for $\R_+\times \R^q\ni (t,y)$ and the Mellin transform in $t$ instead of the Fourier transform. Let
$\mathcal{W}^{s,\gamma_2}(\R_+\times\R^q,H)$ defined to be the completion of $C_0^\infty(\R_+\times\R^q,H)$
with respect to the norm
\[\Vert u\Vert_{\mathcal{W}^{s,\gamma_2}(\R_+\times\R^q,H)}=\Big\{\int_{\Gamma_{(b+1)/2-\gamma_2}}\int_{\R^q}
\langle w,\eta\rangle^{2s}\Vert\varkappa^{-1}_{\langle w,\eta\rangle}(M_{t\to w}\mathcal{F}_{y\to\eta} u)(w,\eta)\Vert^2_{H}\dbar wd\eta\Big\}^{1/2}\]
for a dimension $b$, known from the context (below $b=n+1+q$).
More generally, analogously as Lemma \ref{1-92} we have the following characterisation.
\begin{Lem}\label{2-16}
Let $\varkappa$ be the trivial group action in the space $H$. Then, denoting the respective version of $\mathcal{W}^{s,\gamma_2}$-spaces
by $\mathcal{W}_1^{s,\gamma_2}$, we have continuous embeddings 
$\mathcal{W}_1^{s+M,\gamma_2}(\R_+\times\R^q,H)\hookrightarrow\mathcal{W}^{s,\gamma_2}(\R_+\times\R^q,H)
\hookrightarrow\mathcal{W}_1^{s-M,\gamma_2}(\R_+\times\R^q,H),\,\mathcal{W}^{s+M,\gamma_2}(\R_+\times\R^q,H)\hookrightarrow\mathcal{W}_1^{s,\gamma_2}(\R_+\times\R^q,H)
\hookrightarrow\mathcal{W}^{s-M,\gamma_2}(\R_+\times\R^q,H),\, s\in \R$, for the same constant $M$ as in Lemma \textup{\ref{1-92}}.
\end{Lem}
Observe that the weight $\gamma_2$ is multiplicative
in the sense that
\begin{equation}\label{1-69}
\mathcal{W}^{s,\gamma_2+\beta}(\R_+\times \R^q,H)=t^\beta\mathcal{W}^{s,\gamma_2}(\R_+\times\R^q,H)
\end{equation}  for any $\beta\in\R$.
Note that when the group action $\varkappa$ is unitary in the space $H$ we have
\begin{equation}
\mathcal{W}^{0,0}(\R_{+,t}\times \R^q,H)=t^{-b/2}L^2(\R_{+,t}\times \R^q,H).
\end{equation}In particular, this is the case for $H=\mathcal{K}^{0,0}(X^\wedge)=r^{-n/2}L^2(\R_{+,r}\times X)$.
Then it follows that
\begin{equation}\label{1-88}
\mathcal{W}^{0,0}(\R_{+,t}\times \R^q,\mathcal{K}^{0,0}(X^\wedge))=t^{-b/2}r^{-n/2}L^2(\R_{+,t}\times \R_{+,r}\times \R^q\times X).
\end{equation}
\begin{Prop}
The transformation \[S_{\gamma_2-b/2}u(t):=e^{((b+1)/2-\gamma_2)\boldsymbol{t}}u(e^{-\boldsymbol{t}}),\,\,
S_{\gamma_2-b/2}: C_0^\infty(\R_+,H)\to C_0^\infty(\R,H),\] extends to an isomorphism
\begin{equation}\label{2-6}\mathcal{W}^{s,\gamma_2}(\R_+\times\R^q,H)\to \mathcal{W}^s(\R\times \R^q,H)\end{equation} for every $s,\gamma_2 \in \R$
\textup{(}concerning the space on the right of the latter relation, cf. Definition \textup{\ref{1-25})}.
\end{Prop} The proof follows in an analogous manner as for \eqref{1-77}, based on the relationship between Fourier and
Mellin transform. \\
\\ Let us now turn to an analogue of the spaces $\mathcal{H}^{s,\gamma}(X^\wedge)$ from Section 1 for a compact manifold
$B$ with smooth edge.
\begin{Def}\label{2-8}
By $\mathcal{H}^{s,(\gamma_1,\gamma_2)}(B^\wedge)$ for any $s,\gamma_1,\gamma_2\in \R$ we denote the completion of $C_0^\infty(\R_+\times (B\setminus Y))$
with respect to the norm
\begin{equation}\label{2-21}\| u\|_{\mathcal{H}^{s,(\gamma_1,\gamma_2)}(B^\wedge)}=\big\{\|(1-\omega)u\|^2_{\mathcal{H}^{s,\gamma_2}((2\B)^\wedge)}
+\sum_{j=1}^N\|\varphi_j\omega u\|^2_{\mathcal{W}^{s,\gamma_2}(\R_+\times \R^q,\mathcal{K}^{s,\gamma_1}(X^\wedge))}\big\}^{1/2}\end{equation}
for any cut-off function $\omega=\omega(r)$.
\end{Def}
Here, for notational convenience $\varphi_j$ is identified with an element in $C_0^\infty(\R^q)$. Moreover, $\omega$ is a cut-off
function on $B$ close to $Y$, in the local splitting of variables $(r,x,y)$ an $\omega(r)$, and the number $b:=\textup{dim}\,B=n+1+q$.
Note that when $a(\tau,\eta)\in S^\mu(\Gamma_{(b+1)/2-\gamma_2}\times \R^q;H,\tilde{H})$
(for Hilbert spaces $H$ and $\tilde{H}$ with group actions $\varkappa$ and $\tilde{\varkappa}$, respectively)
we obtain continuous operators
\[\textup{op}_{M_t}^{\gamma_2-b/2}\textup{Op}_y(a):\mathcal{W}^{s,\gamma_2}(\R_+\times\R^q,H)\to \mathcal{W}^{s-\mu,\gamma_2}(\R_+\times\R^q,\tilde{H})\]
for every $s\in \R$. 
Recall from \eqref{0-6} in the case $2\B$ rather than $X$ that we have
$\mathcal{H}^{0,0}((2\B)^\wedge)=t^{-b/2}L^2(\R_+\times 2\B).$ Moreover,
since $\varkappa$ is unitary on the space $\mathcal{K}^{0,0}(X^\wedge)$ we
have the relation \eqref{1-88}. Thus $u\in \mathcal{H}^{0,(0,0)}(B^\wedge)$
means that
\begin{equation}\label{2-20}
\|u\|_{\mathcal{H}^{0,(0,0)}(B^\wedge)}=\big\{\|(1-\omega)u\|^2_{t^{-b/2}L^2(\R_+\times 2\B)}\\
+\sum_{j=1}^N\|\varphi_j\omega u\|^2_{t^{-b/2}r^{-n/2}L^2(\R_+\times \R_+\times \R^q\times X)}\big\}^{1/2}
\end{equation} is finite. The norm in $\mathcal{H}^{0,(0,0)}(B^\wedge)$ is generated by a corresponding
scalar product
\begin{equation}\label{1-89}
\begin{split}
(u,v)_{\mathcal{H}^{0,(0,0)}(B^\wedge)}&=((1-\omega)u,(1-\omega)v)_{t^{-b/2}L^2(\R_+\times 2\B)}\\
&+\sum_{j=1}^N(\varphi_j\omega u,\varphi_j\omega v)_{t^{-b/2}r^{-n/2}L^2(\R_+\times \R_+\times \R^q\times X)}.
\end{split}
\end{equation}
\begin{Rem}
The transformations
\begin{equation}\label{2-9}
\psi_\lambda:u(t,\cdot)\to \lambda^{\frac{b+1}{2}}u(\lambda t,\cdot),\quad \lambda\in \R_+,
\end{equation} define a strongly continuous group of isomorphisms
$\psi_\lambda: \mathcal{H}^{s,(\gamma_1,\gamma_2)}(B^\wedge) \to \mathcal{H}^{s,(\gamma_1,\gamma_2)}(B^\wedge)$
for any $s, \gamma_1, \gamma_2\in \R$. The operators $\psi_\lambda$ are unitary on $\mathcal{H}^{0,(0,0)}(B^\wedge)$.
\end{Rem}
\begin{Rem}\label{2-37}
The scalar product in $\mathcal{H}^{0,(0,0)}(B^\wedge)$ extends
$(\cdot,\cdot):C_0^\infty(\R_+\times(B\setminus Y))\times C_0^\infty(\R_+\times (B\setminus Y))\to \C$
to a non-degenerate sesquilinear pairing
$$\mathcal{H}^{s,(\gamma_1,\gamma_2)}(B^\wedge)\times \mathcal{H}^{-s,(-\gamma_1,-\gamma_2)}(B^\wedge)\to \C$$
for every $s,\gamma_1,\gamma_2\in \R$.
\end{Rem}
\begin{Thm}\label{2-39}
For any real $s_0\leq s_1$, $\gamma_{1,0}\leq \gamma_{1,1}$, $\gamma_{2,0}\leq \gamma_{2,1}$
and $0\leq \theta\leq 1$ we have
\[[\mathcal{H}^{s_0,(\gamma_{1,0},\gamma_{2,0})}(B^\wedge),\mathcal{H}^{s_1,(\gamma_{1,1},\gamma_{2,1})}(B^\wedge)]_\theta=
\mathcal{H}^{s,(\gamma_1,\gamma_2)}(B^\wedge)\] for $s=(1-\theta)s_0+\theta s_1$, $\gamma_1=(1-\theta)\gamma_{1,0}+\theta \gamma_{1,1}$,
 $\gamma_2=(1-\theta)\gamma_{2,0}+\theta \gamma_{2,1}$.
\end{Thm}
\begin{proof}
The arguments are analogous to those in \cite[Section 2.1.2, 4.2.1]{Kapa10}, see also \cite{Hirs2}.
\end{proof}
\begin{Prop}\label{2-17}
For every $s\in \N$ we have the following equivalence of norms
\begin{align}
&\| u\|_{\mathcal{W}^{s,\gamma_2}(\R_+\times \R^q,\mathcal{K}^{s,\gamma_1}(X^\wedge))} \sim\{
 \| u\|^2_{\mathcal{W}^{0,\gamma_2}(\R_+\times \R^q,\mathcal{K}^{0,\gamma_1}(X^\wedge))}\notag\\
&+\sum_{|\beta|=s}\|(t\partial_t)^{\beta^\prime} D_y^{\beta^{\prime\prime}}u\|^2_{\mathcal{W}^{0,\gamma_2}(\R_+\times \R^q,
\mathcal{K}^{0,\gamma_1}(X^\wedge))}\notag\\&+\sum_{|\alpha|=s, 1\leq l\leq N}\| D_l^\alpha u\|^2_{\mathcal{W}^{0,\gamma_2}(\R_+\times \R^q,
\mathcal{K}^{0,\gamma_1;-s}(X^\wedge))}\notag\\&+\sum_{|\alpha|=s, 1\leq l\leq N,|\beta|=s}
\|(t\partial_t)^{\beta^\prime} D_y^{\beta^{\prime\prime}}D_l^\alpha u\|^2_{\mathcal{W}^{0,\gamma_2}(\R_+\times \R^q,
\mathcal{K}^{0,\gamma_1;-s}(X^\wedge))}\}^{1/2}
\end{align} for the operators $D_l^{\alpha}$ from Proposition \textup{\ref{1-79}}. In the sum over $\alpha$ we may
equivalently take $0<|\alpha|\leq s$, and $\|\cdot\|_{\mathcal{W}^{0,\gamma_2}(\R_+\times \R^q,
\mathcal{K}^{0,\gamma_1;-|\alpha|}(X^\wedge))}$ and $0< \beta\leq s$ in the sum over $\beta$.
\end{Prop}
\begin{proof}
By virtue of \eqref{2-6} we mainly have to rephrase $\| S_{\gamma_2-n_1/2} u\|_{\mathcal{W}^s(\R\times \R^q,H)}$
for $H=\mathcal{K}^{s,\gamma_1}(X^\wedge)$. From Proposition \ref{1-84} we have the equivalence of norms
\begin{align}\label{2-7}
&\| v\|_{\mathcal{W}^s(\R^{1+q},\mathcal{K}^{s,\gamma_1}(X^\wedge))}
\sim \big\{\| v\|^2_{\mathcal{W}^0(\R^{1+q},\mathcal{K}^{0,\gamma_1}(X^\wedge))}
\notag\\&+\sum_{|\beta|=s}\|(\partial_{\tilde{t}} +d)^{\beta^\prime}D_y^{\beta^{\prime\prime}} v\|^2_{\mathcal{W}^0(\R^{1+q},\mathcal{K}^{0,\gamma_1}(X^\wedge))}
\notag\\&+\sum_{|\alpha|=s,1\leq l\leq N}\|D_l^\alpha v\|^2_{\mathcal{W}^0(\R^{1+q},\mathcal{K}^{0,\gamma_1;-s}(X^\wedge))}\notag\\&+\sum_{|\alpha|=s,1\leq l\leq N,|\beta|=s}\|(\partial_{\tilde{t}} +d)^{\beta^\prime}D_y^{\beta^{\prime\prime}}D_l^\alpha v\|_{\mathcal{W}^0(\R^{1+q},\mathcal{K}^{0,\gamma_1;-s}(X^\wedge))}\big\}^{1/2}
\end{align} for the constant $d=(n_1+1)/2-\gamma_2$. Inserting $v:=S_{\gamma_2-n_1/2}u$ and applying $\partial_t \circ S_\gamma=S_\gamma\circ (-r\partial_r-(1/2-\gamma))$
gives us the result.
\end{proof}
\begin{Cor}
For every $s\in \N$, $\gamma_1,\gamma_2\in \R$ the relation $u(t,\cdot)\in \mathcal{H}^{s,(\gamma_1,\gamma_2)}(B^\wedge)$
is equivalent to
$$(t\partial_t)^kD_{\boldsymbol{x}}^\alpha(1-\omega)u(t,\boldsymbol{x})\in t^{-b/2+\gamma_2}L^2(\R_+\times 2\B)$$
$k+|\alpha|\leq s$, for all $D^\alpha_{\boldsymbol{x}}\in \textup{Diff}^\alpha(2\B)$ with $\boldsymbol{x}$ being the
variable on the smooth manifold $2\B$, cf. Proposition \textup{\ref{1-91}}, together with the conditions
\[\varphi_j\omega u\in \mathcal{W}^{0,\gamma_2}(\R_+\times \R^q,\mathcal{K}^{0,\gamma_1}(X^\wedge)),\]
\[(t\partial_t)^{\beta^\prime}D_y^{\beta^{\prime\prime}}(\varphi_j\omega u)\in \mathcal{W}^{0,\gamma_2}(\R_+\times \R^q,\mathcal{K}^{0,\gamma_1}(X^\wedge))\]
for all $\beta=(\beta^\prime,\beta^{\prime\prime})$, $|\beta|=s$,
\[D_l^\alpha(\varphi_j\omega u)\in \mathcal{W}^{0,\gamma_2}(\R_+\times\R^q,\mathcal{K}^{0,\gamma_1;-s}(X^\wedge))\]
for all $|\alpha|=s$, $1\leq l\leq N$, and for all $j$, where $\omega=\omega(r)$, and 
\[(t\partial_t)^{\beta^\prime}D_y^{\beta''}D_l^\alpha(\varphi_j\omega u)\in \mathcal{W}^{0,\gamma_2}(\R_+\times\R^q,
\mathcal{K}^{0,\gamma_1;-s}(X^\wedge))\] for all $|\alpha|=s$, $1\leq l\leq N$, and for all $j$, $|\beta|=s$.
\end{Cor}

Let us now consider the case of spaces over $B^\asymp=\R\times B\ni (t,\cdot)$ for a compact manifold
$B$ with edge, where $B^\asymp$ is regarded as a space with conical exits $t\to \pm\infty$. Later on we are
only interested in the positive side. We intend to compare the spaces $\mathcal{H}^{s,(\gamma_1,\gamma_2)}(B^\wedge)$
with spaces of the kind $H^{s,\gamma_1}_{\textup{cone}}(B^\wedge):=H^{s,\gamma_1}_{\textup{cone}}(B^\asymp)|_{\R_+\times B}$,
first for $s\in \N$. To this end we first give the definition of $H^{s,\gamma_1}_{\textup{cone}}(B^\asymp)$. Close to
the edge of $B$ we may model the space on $H^s_{\textup{cone}}((2\B)^\asymp)$ according to the corresponding construction
for a smooth compact cross section. Far from the edge
of $B$ the global definition of $H^{s,\gamma_1}_{\textup{cone}}(B^\asymp)$ relies on a corresponding notion of spaces
$\mathcal{W}^{s,\gamma_1}_{\textup{cone}}(\R\times\R^q\times X^\wedge)$ in the variables $(t,r,x,y)$, where $t$ replaces
the former $r$ while $(y,r,x)$ are variables on the singular cross section $B$. Analogously as \eqref{1-85} we form
$$(\delta_{[t]}f)(t,r,x,y):=f(t,[t]r,x,[t]y)$$ and define
\begin{equation}\label{1-87}\mathcal{W}^{s,\gamma_1}_{\textup{cone}}(\R\times\R^q\times X^\wedge)=\delta_{[t]}\mathcal{W}^{s,\gamma_1}(\R\times\R^q\times X^\wedge)
=\delta_{[t]}\mathcal{W}^{s}(\R\times\R^q,\mathcal{K}^{s,\gamma_1}(X^\wedge)),
\end{equation}
cf. also \eqref{1-86}, in other words
\[\mathcal{W}^{s,\gamma_1}_{\textup{cone}}(\R\times\R^q\times X^\wedge)=\{u(t,r,x,y):u(t,[t]^{-1}r,x,[t]^{-1}y)\in
\mathcal{W}^s(\R^{1+q}_{t,y},\mathcal{K}^{s,\gamma_1}(X_{r,x}^\wedge))\},\] $s\in \R$. In the intersection
zone of the localisation close and off the edge both definitions are equivalent.
In other words, we have an isomorphism
\[\delta_{[t]}^{-1}:\mathcal{W}^{s,\gamma_1}_{\textup{cone}}(\R\times\R^q\times X^\wedge)\to \mathcal{W}^s(\R^{1+q},\mathcal{K}^{s,\gamma_1}(X^\wedge))\]
where $ \|u\|_{\mathcal{W}^{s,\gamma_1}_{\textup{cone}}(\R\times\R^q\times X^\wedge)}=\|\delta_{[t]}^{-1} u\|_{ \mathcal{W}^s(\R^{1+q},\mathcal{K}^{s,\gamma_1}(X^\wedge))}.$
Proposition \ref{1-84}
gave us a characterisation of $\mathcal{W}^s(\R^{1+q}_{t,\tilde{y}},\mathcal{K}^{s,\gamma_1}(X_{\tilde{r},x}^\wedge))$ for $s\in \N$, namely, that
$u(t,\tilde{r},x,\tilde{y})\in \mathcal{W}^s(\R^{1+q}_{t,\tilde{y}},\mathcal{K}^{s,\gamma_1}(X_{\tilde{r},x}^\wedge))$ is equivalent to
$u\in \mathcal{W}^0(\R^{1+q},\mathcal{K}^{0,\gamma_1}(X^\wedge))$ together with
$$D^\beta_{t,\tilde{y}} u\in \mathcal{W}^0(\R^{1+q},\mathcal{K}^{0,\gamma_1}(X^\wedge)),\,\,|\beta|=s,$$ and
$$D^\alpha_{j,(\tilde{r},x)} u\in \mathcal{W}^0(\R^{1+q},\mathcal{K}^{0,\gamma_1;-s}(X^\wedge)),\,\,0\leq j\leq N, |\alpha|=s.$$
\begin{Prop}\label{2-18}
For $s\in \N$, we have
\begin{align}
&\|u\|_{\mathcal{W}^{s,\gamma_1}_{\textup{cone}}(\R\times \R^q\times X^\wedge)}
\sim\big\{\|u(t,r[t]^{-1},x,y[t]^{-1})\|^2_{\mathcal{W}^0(\R\times \R^q,\mathcal{K}^{0,\gamma_1}(X^\wedge))}\notag
\\ &+\sum_{|\beta|=s}\|D_t^{\beta^\prime}D_y^{\beta^{\prime\prime}}
u(t,r[t]^{-1},x,y[t]^{-1})\|^2_{\mathcal{W}^0(\R\times \R^q,\mathcal{K}^{0,\gamma_1}(X^\wedge))}\notag
\\ &+\sum_{|\alpha|=s,1\leq j\leq N}\|D_j^\alpha
 u(t,r[t]^{-1},x,y[t]^{-1})\|^2_{\mathcal{W}^0(\R\times \R^q,\mathcal{K}^{0,\gamma_1;-s}(X^\wedge))}\notag
 \\&+\sum_{|\alpha|=s,1\leq j\leq N,|\beta|=s}\|D_t^{\beta^\prime}D_y^{\beta^{\prime\prime}}D_j^\alpha
 u(t,r[t]^{-1},x,y[t]^{-1})\|^2_{\mathcal{W}^0(\R\times \R^q,\mathcal{K}^{0,\gamma_1;-s}(X^\wedge))}\big\}^{1/2}
\end{align}
\end{Prop}In the sum over $\alpha$ we may
equivalently take $0<|\alpha|\leq s$, and
$\|\cdot\|_{\mathcal{W}^0(\R\times \R^q,\mathcal{K}^{0,\gamma_1;-|\alpha|}(X^\wedge))}$ and $0< \beta\leq s$ in the sum over $\beta$.
\begin{proof}
In fact, because of \eqref{1-87}, it suffices to apply
Proposition \ref{1-84} to $(\delta_{[t]}^{-1} u)(t,r,x,y)=u(t,r[t]^{-1},x,y[t]^{-1})$.
\end{proof}
\begin{Def}\label{1-sp}
\begin{enumerate}
\item[\textup{(i)}]By $H^{s,\gamma _1}_{\textup{cone}}(B^\asymp )$ for $s,\gamma _1\in \R$ we denote the completion of $C_0^\infty (\R\times (B\setminus Y))$ with respect to the norm
\begin{equation}\label{1-speqe}
\| u\| _{H^{s,\gamma _1}_{\textup{cone}}(B^\asymp )}=\big\{\|(1-\omega ) u\| ^2_{H^s_{\textup{cone}}((2\B)^\asymp )}+\sum_{j=1}^N \|\omega  \varphi _j u\| ^2_{\mathcal{W}^{s,\gamma _1}_{\textup{cone}}(\R\times \R^q\times X^\wedge)}\big\}^{1/2}
\end{equation} for any cut-off function $\omega=\omega(r)$.
Moreover, we set
\[H^{s,\gamma_1;e}_{\textup{cone}}(B^\asymp):=\langle t\rangle^{-e}H^{s,\gamma_1}_{\textup{cone}}(B^\asymp),\qquad H^{s,\gamma_1;e}_{\textup{cone}}(B^\wedge)=H^{s,\gamma_1;e}_{\textup{cone}}(B^\asymp)|_{B^\wedge},\]
for any $e\in \R$, then $H^{s,\gamma_1}_{\textup{cone}}(\cdot)=H^{s,\gamma_1;0}_{\textup{cone}}(\cdot)$.
\item[\textup{(ii)}]
The space $\mathcal{K}^{s,\gamma}(B^\wedge)$ for $s\in \R$, $\gamma=(\gamma_1,\gamma_2)\in \R^2$,
 is defined as $$\mathcal{K}^{s,\gamma}(B^\wedge):=\{\omega u_0+(1-\omega)u_\infty: u_0\in \mathcal{H}_{\textup{cone}}^{s,\gamma}(B^\wedge),
 u_\infty\in {H}^{s,\gamma_1}(B^\wedge)\}$$ for any cut-off function $\omega(t)$.
 Moreover, we set $\mathcal{K}^{s,\gamma;e}(B^\wedge):=\langle t\rangle^{-e}\mathcal{K}^{s,\gamma}(B^\wedge)$ for any $e \in \R$.
\end{enumerate}
\end{Def}
\begin{Rem}\label{2-40}
\begin{enumerate}
\item[\textup{(i)}] Observe that the space $\mathcal{K}^{\infty,\gamma;\infty}(B^\wedge)=\bigcap_{s,e\in \R}\mathcal{K}^{s,\gamma;e}(B^\wedge)$
does not depend on the choice of the involved group action $\varkappa$. Also the spaces $\mathcal{H}^{\infty,\gamma}(B^\wedge)$
and $H^{\infty,\gamma}_{\textup{cone}}(B^\asymp)$ \textup{(}the corresponding intersections over $s\in \R$\textup{)} are independent
of $\varkappa$.
\item[\textup{(ii)}] The operators of multiplication by $\omega(r)$ and $1-\omega(r)$ for some cut-off function $\omega$ \textup{(}that are also
globally defined on $\B$\textup{)} induce continuous operators in the spaces $\mathcal{K}^{\infty,\gamma;\infty}(B^\wedge).$ 
\end{enumerate}
\end{Rem} For (i), it suffices to employ Lemma \ref{1-92}, applied to the $\mathcal{W}$-spaces involved in Definition \ref{2-8} and \ref{1-sp}.
For (ii) we may omit $e$ and use $\omega\mathcal{K}^{\infty,\gamma}(B^\wedge)=\bigcap_{s\in \N}\omega\mathcal{K}^{s,\gamma}(B^\wedge)=
\bigcap_{s\in \N}\omega\mathcal{K}_1^{s,\gamma}(B^\wedge)$ where $\omega\mathcal{K}_1^{s,\gamma}(B^\wedge)$ is defined in an analogous manner 
as $\bigcap_{s\in \N}\omega\mathcal{K}^{s,\gamma}(B^\wedge)$ where only the $\mathcal{W}$-spaces involved in Proposition \ref{2-17} and \ref{2-18}
are to be replaced by respective $\mathcal{W}_1$-spaces with $1$ indicating the trivial group action in the $\mathcal{K}$-spaces over $X^\wedge$.
Then the $1$-analogues of Propositions \ref{2-17} and \ref{2-18} give us characterisations of $\omega\mathcal{K}_1^{s,\gamma}(B^\wedge)$ for 
$s\in \N$ in $\mathcal{W}_1^{0,\gamma_2}(\R_+\times \R^q,...)$ and $\mathcal{W}_1^{0}(\R\times \R^q,...)$ norms, and those can easily be
evaluated in weighted $L^2$ norms in the variables $(t,r,y),$ cf. also the proof Theorem \ref{1-18} below. Finally for 
the multiplication by $1-\omega$ the arguments are simpler since we are then off the edge $\R_+\times Y$ of $B^\wedge$.
\begin{Rem}\label{1-speq}
By notation the $\mathcal{W}^{s,\gamma _1}_{\textup{cone}}(\R\times \R^q\times X^\wedge)$-norms in the second term on the right of \eqref{1-speqe} can equivalently be written as
$\|\delta ^{-1}_{[t]}(\omega  \varphi _j u)\| _{\mathcal{W}^s(\R\times \R^q,\mathcal{K}^{s,\gamma _1} (X^\wedge))} .$
\end{Rem}
Also for the cone spaces we need an $L^2$-characterisation of $\mathcal{H}^{0,0}_{\textup{cone}}(B^\asymp)$. From Definition \ref{1-sp} it follows
that $u\in \mathcal{H}^{0,0}_{\textup{cone}}(B^\asymp)$ is equivalent to $(1-\omega)u\in [t]^{-b/2}L^2(\R\times 2\B)$
for a cut-off function $\omega=\omega(r)$,
together with
$\omega\varphi_j u\in [t]^{-b/2}\mathcal{W}^{0,0}_{\textup{cone}}(\R\times \R^q\times X^\wedge)$
for any $j$, or \[\delta^{-1}_{[t]}(\omega\varphi_j u)\in \mathcal{W}^0(\R\times \R^q,\mathcal{K}^{0,0}(X^\wedge))
=r^{-n/2}[t]^{-b/2}L^2(\R\times \R^q\times \R_+\times X).\]
Here we employed again that $\varkappa$ is unitary on $\mathcal{K}^{0,0}(X^\wedge)$. In other words,
$u\in H^{0,0}_{\textup{cone}}(B^\asymp)$ means that
\[\|u\|_{H^{0,0}_{\textup{cone}}(B^\asymp)}=\big\{\|(1-\omega)u\|^2_{[t]^{-b/2}L^2(\R\times 2\B)}+\sum_{j=1}^N\|\delta^{-1}_{[t]}(\omega
\varphi_j u)\|^2_{r^{-n/2}L^2(\R\times \R_+\times \R^q\times X)}\big\}^{\frac{1}{2}}\] for $\omega=\omega(r)$; again equivalent norms are denoted by the same letter.
Note that over the intersection of the supports of $\omega$ and $1-\omega$ there is no contradiction with respect to the behaviour for $|t|\to \infty$.
The norm in $H^{0,0}_{\textup{cone}}(B^\asymp)$ is generated by a scalar product, analogously as explained in connection with \eqref{1-89}, namely,
\begin{multline}\label{2-33}
(u,v)_{H^{0,0}_{\textup{cone}}(B^\asymp)}=((1-\omega)u,(1-\omega)v)_{[t]^{-b/2}L^2(\R\times 2\B)}\\+\sum_{j=1}^N(\delta_{[t]}^{-1}(\omega\varphi_j u),
\delta_{[t]}^{-1}(\omega\varphi_j v))_{r^{-n/2}L^2(\R\times\R_+\times \R^q\times X)}.\end{multline}
\begin{Rem}
The transformations \eqref{2-9} define a strongly continuous groups of isomorphisms
$\psi_\lambda: H^{s,\gamma_1}_{\textup{cone}}(B^\asymp) \to H^{s,\gamma_1}_{\textup{cone}}(B^\asymp)$
for any $s,\gamma_1 \in \R$. The operators $\psi_\lambda$ are unitary in $H^{0,0}_{\textup{cone}}(B^\asymp)$.
\end{Rem}
\begin{Rem}\label{2-38}
The scalar product in $H^{0,0}_{\textup{cone}}(B^\asymp)$ extends
$(\cdot,\cdot): C_0^\infty(\R\times (B\setminus Y))\times C_0^\infty(\R\times (B\setminus Y))\to \C$ to a non-degenerate sesquilinear  pairing
\[H^{s,\gamma_1}_{\textup{cone}}(B^\asymp)\times H^{-s,-\gamma_1}_{\textup{cone}}(B^\asymp)\to \C\]
for every $s,\gamma_1\in \R$.
\end{Rem}
\begin{Thm}\label{2-34}
For any real $s_0\leq s_1$, $\gamma_{1,0}\leq \gamma_{1,1}$ and $0\leq \theta \leq 1$ we have
\[[H^{s_0,\gamma_{1,0}}_{\textup{cone}}(B^\asymp),H^{s_1,\gamma_{1,1}}_{\textup{cone}}(B^\asymp)]_{\theta}=H^{s,\gamma_1}_{\textup{cone}}(B^\asymp)\]
for $s=(1-\theta)s_0+\theta s_1$, $\gamma_1=(1-\theta)\gamma_{1,0}+\theta \gamma_{1,1}$.
\end{Thm}
\begin{proof}
The assertion can be reduced again to the known interpolation properties of the space involved in Definition \ref{2-8}.
\end{proof}
\begin{Def}
The space $R^m_G(\Sigma\times
\R^d)_{\mathcal{O}}$ is defined to
be the set of all operator functions $g(v,\zeta)$ such that
$$g(v,\zeta)\in S_{\textup{cl}}^m(\Sigma\times
\R^d;\mathcal{K}^{s,(\gamma_1,\gamma_2);e}(B^\wedge),\mathcal{K}^{s^\prime,(\gamma_1^\prime,\gamma_2^\prime);e^\prime}(B^\wedge)),$$
$$g^*(v,\zeta)\in S_{\textup{cl}}^m(\Sigma\times
\R^d;\mathcal{K}^{s,(\gamma_1,\gamma_2);e}(B^\wedge),\mathcal{K}^{s^\prime,(\gamma_1^\prime,\gamma_2^\prime);e^\prime}(B^\wedge))$$
for all $s,s^\prime,\gamma_1,\gamma_1^\prime,\gamma_2,\gamma_2^\prime,e,e^\prime\in \R.$ Such a $g$ is also called a flat Green symbol
of order $m$ in the second order corner calculus. 
\end{Def}
\begin{Rem}
It can easily be proved that for cut-off functions $\omega(r)$, $\omega^\prime(r)$ on $B$
and any $g\in R^m_G(\Sigma\times \R^d)_\mathcal{O}$, also $\omega g\omega^\prime, (1-\omega)g\omega^\prime, \omega g(1-\omega^\prime), 
(1-\omega)g(1-\omega^\prime)$ belong to $ R^m_G(\Sigma\times \R^d)_\mathcal{O},$ cf. also Remark \textup{\ref{2-40} (ii)}.
\end{Rem}

\section{The Mellin-edge quantisation}
We now turn to the Mellin-edge quantisation for operator-valued edge symbols along an edge of second singularity order.
\begin{Def}\label{1-4}
The space of holomorphic \textup{(}parameter-dependent\textup{)} Mellin symbols
$M^m_{\mathcal{O}}(B,\boldsymbol{g};\R^l)$ is
defined to be the set of all $h(w,\lambda)\in
\mathcal{A}(\C,L^m(B,\boldsymbol{g};\R^l))$ such
that $h|_{\Gamma_\delta}\in
L^m(B,\boldsymbol{g};\Gamma_\delta\times\R^l)$ for
every $\delta\in \R$, uniformly in compact $\delta$-intervals.
\end{Def}
The following theorem is a Mellin quantisation result.
\begin{Thm}\label{1-5}
For every $p(t,v,\tau,\zeta)=\tilde{p}(t,v,t\tau,t\zeta)$,
$\tilde{p}(t,v,\tilde{\tau},\tilde{\zeta})\in
C^\infty(\overline{\R}_+\times
\Sigma,L^m(B,\boldsymbol{g};\R^{1+d}_{\tilde{\tau},\tilde{\zeta}}))$,
$\Sigma\subseteq \R^d$ open, there exists an
$h(t,v,w,\zeta)=\tilde{h}(t,v,w,t\zeta)$,
$\tilde{h}(t,v,w,\tilde{\zeta})\in C^\infty(\overline{\R}_+\times
\Sigma,M_{\mathcal{O}}^m(B,\boldsymbol{g};\R^{d}_{\tilde{\zeta}}))$
such that
$\textup{Op}_t(p)(v,\zeta)=\textup{op}^\beta_{M_t}(h)(v,\zeta)$
\textup{mod}
$C^\infty(\Sigma,L^{-\infty}(B^\wedge,\boldsymbol{g};\R^{d}_{{\zeta}}))$
for every $\beta\in \R$; the operators are regarded first
as continuous maps
$C^\infty_0(\R_+,H^{\infty,\gamma_1}(B)) \to
C^\infty(\R_+,H^{\infty,\gamma_1-\mu}(B))$.
\end{Thm} 
Concerning the proof and more details, cf. \cite{Schu27},
\cite{Haru11}.
More precisely, we can choose $h$ in such a way that for some
$\psi \in C_0^\infty(\R_+)$ which is equal to $1$ close to $1$ we
have \begin{equation}\label{1-9}
\textup{Op}_t(p)(v,\zeta)-\textup{op}_M^{\gamma_2-\frac{b}{2}}(h)(v,\zeta)=\textup{Op}_t((1-\psi(\frac{t^\prime}{t}))p)(v,\zeta)
\end{equation} Set $p_0(t,v,\tau,\zeta):=\tilde{p}(0,v,t\tau,t\zeta)$,
$h_0(t,v,\tau,\zeta):=\tilde{h}(0,v,t\tau,t\zeta)$. Note that then
we also have
$\textup{Op}_t(p_0)(v,\zeta)=\textup{op}_M^\beta(h_0)(v,\zeta)$
mod
$C^\infty(\Sigma,L^{-\infty}(B^\wedge,\boldsymbol{g};\R^{d}_{{\zeta}}))$
for every $\beta\in\R$.
We now formulate the non-smoothing part of edge
symbols for the calculus over a (stretched) wedge $B^\wedge\times
\Sigma$. Similarly as in the case $X^\wedge\times \Omega$, cf.
Section 1.1, we fix cut-off functions $\omega$, $\omega^\prime$
and excision functions $\chi$, $\chi^\prime$ on the $t$ half-axis
such that $\omega\prec \omega^\prime $, $\chi\prec \chi^\prime$,
$\omega+\chi=1$ and cut-off functions $\epsilon$, $\epsilon^\prime$. Then we form
\begin{equation}\label{2-1}
a(v,\zeta):=t^{-\mu}\epsilon\{\omega_\zeta\textup{op}_M^{\gamma_2-\frac{b}{2}}
(h)(v,\zeta)\omega_\zeta^\prime+\chi_\zeta\textup{Op}_t(p)(v,\zeta)\chi_\zeta^\prime\}\epsilon^\prime+r(v,\zeta).
\end{equation} Here $b:=\textup{dim}\,B$, the operator families
$p$ and $h$ are as in Theorem \ref{1-5}, and $r(v,\zeta)\in
C^\infty(\Sigma,L^\mu(B^\wedge,\boldsymbol{g};\R^d_\zeta))$
is a family localised on a compact subset $K$ of $\R_+$ (i.e.
$\varphi r(v,\zeta)\varphi^\prime=0$ for all
$\varphi,\varphi^\prime \in C_0^\infty(\R_+)$ vanishing in a
neighbourhood of $K$). Similarly as in the case of a manifold with smooth edge
it can be proved that $$a(v,\zeta)\in S^\mu(\Sigma\times \R^d;\mathcal{K}^{s,(\gamma_1,\gamma_2)}(B^\wedge),
\mathcal{K}^{s-\mu,(\gamma_1-\mu,\gamma_2-\mu)}(B^\wedge)).$$\\
The operator functions $a(v,\zeta)$ belong to
$L^\mu(B^\wedge,\boldsymbol{g};\R^d_\zeta)$ for
every $v\in \Sigma$. As such they have the symbolic structure from
the calculus over $B^\wedge$, namely,
\begin{equation}\label{2-2}
\sigma_0(a) \quad \textup{and} \quad \sigma_1(a)
\end{equation}
and so-called reduced symbols. To be more precise, \eqref{2-2} are
the usual (parameter-dependent) principal symbols of edge
operators, where $\sigma_i$ refers to $s_i(B)$, $i=0,1$. In
addition, in the splitting of variables $(t,v)$ close to $t=0$ the
symbols occur in the form $t^{-\mu}\sigma_0(a)(t,v,t\tau,t\zeta)$
and $t^{-\mu}\sigma_1(t,v,t\tau,t\zeta)$ for so-called reduced
symbols $\tilde{\sigma}_0(a)(t,v,\tilde{\tau},\tilde{\zeta})$ and
$\tilde{\sigma}_1(a)(t,v,\tilde{\tau},\tilde{\zeta})$,
respectively, which are smooth up to $t=0$. Moreover, close to
$s_1(B)$ in the splitting of variables $(r,x,y)\in X^\wedge\times
\Omega$ and the covariables $(\rho,\xi,\eta)$ for the $0$-th
component we have a degenerate behaviour
$$r^{-\mu}\tilde{\sigma}_0(t,r,x,y,v,r\tilde{\tau},r\rho,\xi,r\eta,r\tilde{\zeta})$$
and there is  another reduced symbol
$\tilde{\tilde{\sigma}}_0(t,r,x,y,v,\tilde{\tilde{\tau}},\tilde{\rho},\xi,\tilde{\eta},\tilde{\tilde{\zeta}})$
which is smooth up to $t=r=0$.\\
\\
The principal edge symbol belonging to the edge $\Sigma$ is
defined as
\begin{equation}\label{2-3}
\sigma_2(a)(v,\zeta):=t^{-\mu}\{\omega(t|\zeta|)\textup{op}_M^{\gamma_2-\frac{b}{2}}(h_0)(v,\zeta)\omega^\prime(t|\zeta|)+
\chi(t|\zeta|)\textup{op}_t(p_0)(v,\zeta)\chi^\prime(t|\zeta|)\},
\end{equation} parametrised by $(v,\zeta)\in T^*\Sigma\setminus
0$. It represents a family of continuous operators
$$\sigma_2(a)(v,\zeta):\mathcal{K}^{s,\gamma}(B^\wedge)\to
\mathcal{K}^{s-\mu,\gamma-\mu}(B^\wedge)$$ for every $s\in \R$,
$\gamma-\mu:=(\gamma_1-\mu,\gamma_2-\mu)$. 

In the edge quantisation expression \eqref{2-1} that quantises the operator function $p(t,v,\tau,\zeta):=\tilde{p}(t,v,t\tau,t\zeta)$ and produces the operator-valued symbol $a(v,\zeta )$ we are mainly interested in the non-smoothing ingredients of the families $p$ and $h,$ respectively. Therefore, before we pass to the alternative Mellin-edge quantised family we split up some unnecessary smoothing terms. In addition because of linearity it is admitted to consider local expressions close to the edge of $B.$ In other words we start with
$$p(t,v,\tau ,\zeta )=\tilde{p}(t,v,t\tau ,t\zeta )$$
for 
$$\tilde{p}(t,v,\tilde{\tau },\tilde{\zeta })\in C^\infty (\overline{\R}_+\times \Sigma ,L^m(B,\boldsymbol{g};\R^d_{\tilde{\zeta }})),$$
where
$$\tilde{p}(t,v,\tilde{\tau },\tilde{\zeta }):=\textup{Op}_y\textup{op}^{\gamma _1-n/2}_{M_r}(r^{-m}\epsilon _1h_1\epsilon '_1)(t,v,\tilde{\tau },\tilde{\zeta })$$
for cut-off functions $\epsilon (r),\epsilon '(r)$ and a Mellin symbol
$$h_1(t,r,y,v,z,\tilde{\tau },\tilde{\zeta })=\tilde{h}_1(t,r,y,v,z,r\tilde{\tau },r\tilde{\zeta }),$$
$$\tilde{h}_1(t,r,y,v,z,\tilde{\tilde{\tau  }},\tilde{\tilde{\zeta }})\in C^\infty (\overline{\R}_+\times \overline{\R}_+\times \R^q\times \Sigma ,M^m_{\mathcal{O}_z}(X,\R^{1+d}_{\tilde{\tilde{\tau  }},\tilde{\tilde{\zeta }}})).$$
\begin{Thm}\label{mainth}
An edge amplitude function
$$a(v,\zeta)=t^{-\mu}\epsilon\{\omega_\zeta
\textup{op}_M^{\gamma_2-\frac{b}{2}}(h)(v,\zeta)\omega_\zeta^\prime+\chi_\zeta\textup{Op}_t(p)(v,\zeta)\chi^\prime_\zeta\}\epsilon^\prime$$
as in \eqref{2-1} can be
expressed in the form
$$a(v,\zeta)=t^{-\mu}\epsilon\textup{op}_M^{\gamma_2-\frac{b}{2}}(h)(v,\zeta)\epsilon^\prime+g(v,\zeta)$$
for some $g(v,\zeta)\in R^\mu_G(\Sigma\times
\R^d)_{\mathcal{O}}$. Then \eqref{2-3} can be written as
$$\sigma_2(a)(v,\zeta)=t^{-\mu}\textup{op}_M^{\gamma_2-\frac{b}{2}}(h_0)(v,\zeta)+\sigma_2(g)(v,\zeta).$$
\end{Thm}
\begin{proof}
For simplicity we assume $a$ to be independent of $v$; the extension to
the general case is straightforward. First we have
\begin{align}\label{2-4}
 &t^{-\mu}\omega_\zeta\textup{op}_M^{\gamma_2-\frac{b}{2}}(h)(\zeta)
\omega^\prime_\zeta+t^{-\mu}\chi_\zeta\textup{Op}_t(p)(\zeta)\chi^\prime_\zeta\\ \notag=&
t^{-\mu}\omega_\zeta\textup{op}_M^{\gamma_2-\frac{b}{2}}(h)(\zeta)
\omega^\prime_\zeta+t^{-\mu}\chi_\zeta\textup{op}_M^{\gamma_2-\frac{b}{2}}(h)(\zeta)\chi^\prime_\zeta
+t^{-\mu}\chi_\zeta\{\textup{Op}_t(p)(\zeta)-\textup{op}_M^{\gamma_2-\frac{b}{2}}(h)(\zeta)\}\chi^\prime_\zeta\\ \notag =&
t^{-\mu}\textup{op}_M^{\gamma_2-\frac{b}{2}}(h)(\zeta)+g_1(\zeta)+g_2(\zeta)
\end{align} where $$g_1(\zeta)=t^{-\mu}\chi_\zeta\{\textup{Op}_t(p)(\zeta)-\textup{op}_M^{\gamma_2-\frac{b}{2}}(h)(\zeta)\}\chi^\prime_\zeta,$$
$$g_2(\zeta)=-t^{-\mu}\chi_\zeta\textup{op}_M^{\gamma_2-\frac{b}{2}}(h)(\zeta)(1-\chi_\zeta^\prime)
-t^{-\mu}\omega_\zeta\textup{op}_M^{\gamma_2-\frac{b}{2}}(h)(\zeta)(1-\omega_\zeta^\prime).$$
The assertion then follows by multiplying both sides of
\eqref{2-4} by $\epsilon$ from the left, $\epsilon^\prime$ from the
right. In fact, $\epsilon g_1(\zeta)\epsilon^\prime\in
R^\mu_G(\R^d)_{\mathcal{O}}$
corresponds to
Proposition \ref{1-7} below and $\epsilon g_2(\zeta)\epsilon^\prime\in
R^\mu_G(\R^d)_{\mathcal{O}}$ to Proposition \ref{1-6} below;
here we use that $\epsilon g_2(\zeta)\epsilon^\prime \in
R^\mu_G(\R^d)_{\mathcal{O}}$.\\
\\
Analogous conclusions are possible for $\zeta\neq 0$ and without
the factors $\epsilon$, $\epsilon^\prime$, i.e. comparing with
\eqref{2-3} we obtain
$\sigma_2(a)(\zeta)=t^{-\mu}\textup{op}_M^{\gamma_2-\frac{b}{2}}(h_0)(\zeta)+g_{1,(\mu)}(\zeta)+g_{2,(\mu)}(\zeta).$
\end{proof}

\section{Characterisation of Green left-over terms}
 The following lemma may be found in \cite{Gil2}.
\begin{Lem}\label{1-14}
Let
$f(t,t^\prime,\zeta):=\omega_\zeta(t)(\textup{log}\frac{t}{t^\prime})^{-N}(1-\omega^\prime_\zeta(t^\prime))$
for cut-off functions $\omega\prec \omega^\prime$, and $t$,
$t^\prime$ $\in \R$, $\zeta\in \R^q$, $N\in \N$, $N\geq 1$. Then
\begin{equation}\label{1-15}
f(\lambda^{-1}t,\lambda^{-1}t^\prime,\lambda\zeta)=f(t,t^\prime,\zeta)
\quad \textup{for}\quad \lambda\geq 1,\,\,\lvert\zeta\rvert\geq
\textup{const},\,\, t,t^\prime\in \R_+,
\end{equation} for some constant $>0$. Moreover, for every $k,$
$k^\prime$ $\in \N$, $\alpha\in \N^q$, we have
\begin{equation}\label{1-16}
\textup{sup}_{t,t^\prime,\zeta}\{\lvert(t\partial_t)^k(t^\prime\partial_{t^\prime})^{k^\prime}\partial_\zeta^\alpha
f(t,t^\prime,\zeta)\rvert[\zeta]^{|\alpha|}\}<\infty.
\end{equation}In particular, the function $(t,t^\prime)\to [\zeta]^{|\alpha|}\partial^\alpha_\zeta
f(t,t^\prime,\zeta)$ is bounded in $\zeta\in \R^q$ as a function of
$\zeta$ with values in $C^\infty(\R_+\times \R_+)_B$.
\end{Lem}
\begin{Prop}\label{1-6}
Let $\tilde{h}(t,v,z,\tilde{\zeta})\in
C^\infty(\overline{\R}_+\times \Sigma,
M_{\mathcal{O}}^m(B,\boldsymbol{g};\R^{d}))$ for $\boldsymbol{g}=(\gamma_1,\gamma_1-\mu,\Theta)$ and
$h(t,v,w,\zeta)=\tilde{h}(t,v,w,t\zeta)$. Let us set
$$g_0(v,\zeta):=\omega_\zeta\textup{op}_M^{\gamma_2-\frac{b}{2}}(h)(v,\zeta)(1-\omega_\zeta^\prime),\,
g_1(v,\zeta):=(1-\omega_\zeta^\prime)\textup{op}_M^{\gamma_2-\frac{b}{2}}(h)(v,\zeta)\omega_\zeta,$$ for cut-off functions
$\omega$, $\omega^\prime$ with $\omega \prec \omega^\prime$. Then
we have $g_0(v,\zeta),g_1(v,\zeta)\in R^0_G(\Sigma\times
\R^d)_{\mathcal{O}}$.
\end{Prop}
\begin{proof}
For simplicity we consider the $v$-independent case. First we show
that the relation
\begin{equation}\label{1-17}
g_0(\zeta)\in
S_\textup{cl}^0(\R^d;\mathcal{K}^{s,(\gamma_1,\gamma_2);g}(B^\wedge),
\mathcal{K}^{s^\prime,(\gamma_1-\mu,\gamma_2^\prime);g^\prime}(B^\wedge))
\end{equation} for arbitrary $s,s^\prime,\gamma_2,\gamma_2^\prime,g,g^\prime\in
\R$. Integration by parts gives us
$g_0(\zeta)=\textup{op}_M^{\gamma_2-\frac{b}{2}}(h_N)(\zeta)$
for every $N\in \N$, for
$h_N(t,t^\prime,w,\zeta)=f(t,t^\prime,\zeta)\partial_w^Nh(t,w,\zeta)$,
where $f(t,t^\prime,\zeta)$ is the function from Lemma \ref{1-14}.
We choose cut-off functions $\omega_1$, $\omega_2$ such that
$\omega_1\prec \omega$, $\omega\prec \omega_2$. Moreover, we
choose $L\geq 0$ so large as in Proposition \ref{1-18} (ii). Since
$h$ is holomorphic in $w$ we obtain
$$g_0(\zeta)=\omega_{2,\zeta}t^L\textup{op}_M^{\gamma^\prime_2-\frac{b}{2}}(f(t,t^\prime,\zeta)T^L\partial^N_wh(t,w,\zeta))
(1-\omega_{1,\zeta})t^{-L}$$ as an operator on
$C_0^\infty(\R_+,H^{\infty,\gamma_1}(B))$. Note that
because of the infinite flatness in $t^\prime$ on the right we
could replace $\textup{op}_{M_t}^{\gamma_2-\frac{b}{2}}$ by
$\textup{op}_{M_t}^{\gamma^\prime_2-\frac{b}{2}}$; then from the
weight $\gamma_2$ we obtain $\gamma_2^\prime$ under the action;
therefore, from now on we argue with $\gamma_2^\prime$ both for the
original function and the one in the image. By virtue of
Proposition \ref{1-18} (ii), (iii) we obtain the relation
\eqref{1-17} if we show that
\begin{equation}\label{1-19}
\begin{split} a(\zeta):=\textup{op}_M^{\gamma_2^\prime-\frac{b}{2}}&(f(t,t^\prime,\zeta)T^L\partial^N_wh(t,w,\zeta))\\ &\in
S^0_{\textup{cl}}(\R^d;\mathcal{H}^{s,(\gamma_1,\gamma_2^\prime)}(B^\wedge),\mathcal{H}^{s-\mu+N,(\gamma_1-\mu+N,\gamma_2^\prime)}(B^\wedge))
\end{split}\end{equation} since $N$ can be chosen so large that $s-\mu+N\geq
s^\prime$. In other words, it remains to verify \eqref{1-19}. To
simplify notation from now on in the proof we replace
$T^L\partial_w^N h$ again by $h$, here associated with
$\tilde{h}(t,w,\tilde{\zeta})\in
C^\infty(\overline{\R}_+,M_\mathcal{O}^{\mu-N}(B,(\gamma_1,\gamma_1-\mu+N,\Theta);\R^d_{\tilde{\zeta}}))$.
Applying the Taylor formula we obtain
$$\tilde{h}(t,w,\tilde{\zeta})=\sum^{J-1}_{j=0}t^j\tilde{h}_j(w,\tilde{\zeta})+t^J\tilde{h}_{J}(t,w,\tilde{\zeta})$$
where $\tilde{h}_j(w,\tilde{\zeta})\in
M_\mathcal{O}^{\mu-N}(B,(\gamma_1,\gamma_1-\mu+N,\Theta);\R^d_{\tilde{\zeta}})$
and $\tilde{h}_{j}(t,w,\tilde{\zeta})\in
M_\mathcal{O}^{\mu-N}(B,(\gamma_1,\gamma_1-\mu+N,\Theta);\R^d_{\tilde{\zeta}})$.
Let $h_j(t,w,\zeta):=\tilde{h}_j(w,t\zeta)$,
$h_J(t,w,\zeta):=\tilde{h}_J(t,w,t\zeta)$. For
$a_j(\zeta):=\textup{op}_M^{\gamma_2^\prime-\frac{b}{2}}(f(t,t^\prime,\zeta)h_j(t,w,\zeta))$
it follows that $a(\zeta)=\sum_{j=0}^{J-1} t^ja_j(\zeta)$. Lemma
\ref{1-14} yields $t^j f(t,t^\prime,\zeta)h_j(t,w,\zeta)\in
C^\infty(\R^d,C^\infty(\R_+\times\R_+,M_\mathcal{O}^{\mu-N}(B,(\gamma_1,\gamma_1-\mu+N,\Theta)))_{\textup{B}})$
for every $j$ (subscript $\rm{B}$ means bounded in
$(t,t^\prime)\in \R_+\times \R_+$ of all derivatives
$(t\partial_t)^l(t^\prime\partial_{t^\prime})^{l^\prime}$ of
the corresponding function). This gives us $t^j a_j(\zeta)\in
C^\infty(\R^d_\zeta,\mathcal{L}(\mathcal{H}^{s,(\gamma_1,\gamma^\prime_2)}(B^\wedge),
\mathcal{H}^{s-\mu+N,(\gamma_1-\mu+N,\gamma^\prime_2)}(B^\wedge)))$.
Moreover, the relation \eqref{1-15} shows $t^j a_j(\zeta)$ for
$0\leq j<J$ is homogenous of order $-j$ for large $\lvert
\zeta\rvert$. We therefore obtain that $t^ja_j(\zeta)\in
S^{-j}_{\textup{cl}}(\R^d;\mathcal{H}^{s,(\gamma_1,\gamma_2^\prime)}(B^\wedge),
\mathcal{H}^{s-\mu+N,(\gamma_1-\mu+N,\gamma_2^\prime)}(B^\wedge))$.
To verify $t^Ja_J(\zeta)=a(\zeta)-\sum_{j=0}^{J-1}t^ja_j(\zeta)\in
S^{-J}(\R_\zeta^d;\mathcal{H}^{s,(\gamma_1,\gamma_2^\prime)}(B^\wedge),
\mathcal{H}^{s-\mu+N,(\gamma_1-\mu+N,\gamma_2^\prime)}(B^\wedge))$,
i.e.
$$\Vert\kappa_{[\zeta]}^{-1}\{D_\zeta^\alpha t^Ja_J(\zeta)\}\kappa_{[\zeta]}\Vert_{\mathcal{L}(\mathcal{H}^{s,(\gamma_1,\gamma_2^\prime)}(B^\wedge),
\mathcal{H}^{s-\mu+N,(\gamma_1-\mu+N,\gamma_2^\prime)}(B^\wedge))}\leq
c[\zeta]^{-J-\vert\alpha\vert}$$ for all $\zeta\in \R^d$,
$\alpha\in \N^d$, we consider
$$\partial_\zeta^\alpha[f(t,t^\prime,\zeta)h_J(t,z,\zeta)]=\sum_{\beta\leq\alpha}
\left(\begin{array}{l}\alpha\\
\beta\end{array}\right)\partial_\zeta^\beta
f(t,t^\prime,\zeta)t^{\vert\alpha-\beta\vert}\partial_\zeta^{\alpha-\beta}\tilde{h}_J(t,z,\zeta)$$
with differentiation referring to the topology of
$C^\infty(\R^d_\zeta,C^\infty(\R_+\times\R_+,M_{\mathcal{O}}^{\mu-N}(B,(\gamma_1,\gamma_1-\mu+N,\Theta)))_{\textup{B}})$.
This gives us
$\kappa^{-1}_{[\zeta]}\partial^\alpha_\zeta(t^Ja_J(\zeta))\kappa_{[\zeta]}[\zeta]^{\vert\alpha\vert+J}
=\textup{op}_{M_t}^{\gamma_2^\prime-\frac{b}{2}}(m_\alpha(t,t^\prime,w,\zeta))$
for $$m_\alpha(t,t^\prime,w,\zeta)=\sum_{\beta\leq\alpha}
\left(\begin{array}{l}\alpha\\
\beta
\end{array}\right)
\partial_\zeta^\beta f(\frac{t}{[\zeta]},\frac{t^\prime}{[\zeta]},\zeta)[\zeta]^\beta
t^{J+\vert\alpha-\beta\vert}(\partial_\zeta^{\alpha-\beta}\tilde{h}_J)(\frac{t}{[\zeta]},w,\frac{t}{[\zeta]}\zeta).$$
Because of Lemma \ref{1-14} the function $m_\alpha$ is bounded in
$\zeta$ with values in $C^\infty(\R_+\times
\R_+,M_{\mathcal{O}}^{\mu-N}(B,(\gamma_1,\gamma_1-\mu+N,\Theta)))_B$
and of compact support with respect to $t$. This yields
$$\textup{sup}_{\zeta\in \R^d}\Vert\textup{op}_M^{\gamma_2^\prime-\frac{b}{2}}(m_\alpha)(\zeta)\Vert_{\mathcal{L}
(\mathcal{H}^{s,(\gamma_1,\gamma_2^\prime)}(B^\wedge),\mathcal{H}^{s-\mu+N,(\gamma_1-\mu+N,\gamma_2^\prime)}(B^\wedge))}<\infty.$$
Hence we obtain altogether the relation \eqref{1-17}. For the
proof the
\begin{equation}\label{1-20}
g_1(\zeta)\in
S_{\textup{cl}}^0(\R^d;\mathcal{K}^{s,(\gamma_1,\gamma_2);e}(B^\wedge),
\mathcal{K}^{s^\prime,(\gamma_1^\prime,\gamma^\prime_2);e^\prime}(B^\wedge))
\end{equation} for all $s,s^\prime,\gamma_1,\gamma_1^\prime,\gamma_2,\gamma_2^\prime,e,e^\prime \in
\R$ we express $\textup{op}_M^{\gamma_2-\frac{b}{2}}(h)$ in the
form
$\textup{op}_M^{\gamma_2-\frac{b}{2}}(h)=\textup{op}_M^{\gamma_2-\frac{b}{2}}(h_R)$
for a right symbol
$h_R(t^\prime,w,\zeta)=\tilde{h}_R(t^\prime,w,t^\prime\zeta)$,
$\tilde{h}_R(t^\prime,w,\tilde{\zeta})\in
C^\infty(\overline{\R}_+,M_{\mathcal{O}}^\mu(B,\boldsymbol{g};\R^d)).$
Then we proceed in an analogous manner as before, i.e. choose an
$L$ such that $\frac{b+1}{2}-\gamma_2+L\geq g^\prime$, and write
$$g_1(\zeta)=(1-\omega_\zeta ^\prime)t^{-L}\textup{op}_M^{\gamma_2-\frac{b}{2}}
((-1)^N f(t,t^\prime,\zeta))T^L\partial_w^N
h_R(t^\prime,w,\zeta)\omega_\zeta t^L.$$ By means of
Proposition \ref{1-18} (i), (iv) and Lemma \ref{1-14}, together
with a Taylor expansion of $\tilde{h}_R$ in $t^\prime$ at
$t^\prime=0$ we can verify the relation \eqref{1-20}. It remains to
observe that for the formal adjoints $g_0^*$, $g_1^*$ we obtain the
same, since $g_0^*$ is analogous as $g_1$ and $g_1^*$ as
$g_0$.
\end{proof}
\begin{Prop}\label{1-7}
Let $p(t,v,\tau,\zeta)$ and $h(t,v,w,\zeta)$ be as in Theorem
\textup{\ref{1-5}}. Then for every two cut-off functions $\epsilon$,
$\epsilon^\prime$ and excision functions $\chi$, $\chi^\prime$ with
$\chi\prec \chi^\prime$ we have
$$g(v,\zeta):=\epsilon\chi_\zeta\{\textup{Op}_t(p)(v,\zeta)-\textup{op}_M^{\gamma_2-\frac{b}{2}}(h)(v,\zeta)\}
\chi^\prime_\zeta\epsilon^\prime\in R_G^0(\Sigma\times
\R^d)_{\mathcal{O}}.$$\end{Prop}
\begin{proof}
For simplicity we consider the $v$-independent case. According to
\eqref{1-9} we have
\begin{equation}\label{1-10}
g(\zeta)u(t)=(1-\omega_\zeta(t))\int\!\!\int
e^{i(t-t^\prime)\tau}(1-\psi(\frac{t^\prime}{t}))\tilde{q}(t,t^\prime,t\tau,t\zeta)(1-\omega_\zeta^\prime(t^\prime))u(t^\prime)dtd\tau\end{equation}
 for
$\tilde{q}(t,t^\prime,\tilde{\tau},\tilde{\zeta}):=\epsilon(t)\epsilon^\prime(t^\prime)\tilde{p}(t,\tilde{\tau},\tilde{\zeta})$,
and cut-off functions $\omega$, $\omega^\prime$ such that
$\chi=1-\omega$, $\chi^\prime=1-\omega^\prime$. First we show that
the (operator-valued) integral kernel of \eqref{1-10} can be
interpreted as a function in
\begin{equation}\label{1-11}
C^\infty(\R_\zeta^d,\mathcal{S}(\R\times\R,L^{-\infty}(B,(\gamma_1,\gamma_1-\mu+N,\Theta))))
\end{equation}\textup{(}here we mean the integral operator on $\R_+\times \R_+$, and functions with compact support
 in $t,t^\prime\in \R_+\times \R_+$ are identified with functions on $\R\times \R$ by extending them by zero in $t\leq 0$ and $t^\prime \leq 0$,
 respectively\textup{)}. Using the identity $(t-t^\prime)^{-N}D_\tau^N
 e^{i(t-t^\prime)\tau}=e^{i(t-t^\prime)\tau}$ for every $N\in\N$
 integration by parts gives us
 \begin{multline}\label{1-12}
 g(\zeta)u(t)=(1-\omega_\zeta(t))\int\!\!\int
 e^{i(t-t^\prime)\tau}(t^\prime-t)^{-N}t^N(1-\psi(\frac{t^\prime}{t}))(D^N_{\tilde{\tau}}\tilde{q})
 (t,t^\prime,t\tau,t\zeta)\\(1-\omega^\prime_\zeta(t^\prime))u(t^\prime)dt^\prime\dbar\tau.
 \end{multline} Since $(D_{\tilde{\tau}}^N\tilde{q})(t,t^\prime,\tilde{\tau},\tilde{\zeta})\in
  C^\infty(\overline{\R}_+\times
  \overline{\R}_+,L^{m-N}(B,(\gamma_1,\gamma_1-\mu+N,\Theta);\R^{1+d}_{\tilde{\tau},\tilde{\zeta}}))$
  for every $N$, the integral kernel of \eqref{1-12} belongs to
  $C^\infty(\R_+\times\R_+,L^{-\infty}(B,(\gamma_1,\gamma_1-\mu+N,\Theta)))$
  but the support properties in $t$, $t^\prime$ show that the
  integral kernel belongs in fact to \eqref{1-11}; the smoothness
  in $\zeta$ is evident. Let $\vartheta(\zeta)$ be an excision
  function. Then \eqref{1-12} together with the fact that the
  kernel is compactly supported in $\R_+\times\R_+$ gives us $$(1-\vartheta(\zeta))g(\zeta)\in
   R_G^{-\infty}(\R^d_\zeta)_{\mathcal{O}}.$$
    Thus it remains to verify $\vartheta(\zeta)g(\zeta)\in
   R^0_G(\R^d)_\mathcal{O}$. Let us set for the
   moment
   $\tilde{p}_0(t\tau,t\zeta):=\tilde{q}(0,0,t\tau,t\zeta)$, and
   $$g_0(\zeta)u(t)=\vartheta(\zeta)(1-\omega_\zeta(t))\int\!\!\int e^{i(t-t^\prime)\tau}
   (1-\psi(\frac{t^\prime}{t}))\tilde{p}_0(t\tau,t\zeta)(1-\omega^\prime_\zeta(t^\prime))u(t^\prime)dt^\prime\dbar\tau.$$
   Then, as before, after integration by parts we have \begin{equation}\label{1-21} g_0(\zeta)u(t)=\vartheta(\zeta)
   \int\!\!\int
   e^{i(t-t^\prime)\tau}f_N(t,t^\prime,\zeta)(D^N_{\tilde{\tau}}\tilde{p_0})(t\tau,t\zeta)u(t^\prime)dt^\prime\dbar\tau\end{equation}
   for $f_N(t,t^\prime,\zeta):=(1-\omega_\zeta(t))(t^\prime-t)^{-N}t^N(1-\psi(\frac{t^\prime}{t}))(1-\omega^\prime_\zeta(t^\prime))$
   for every $N\in \N$. Since $N$ is arbitrary, the kernel of
   $g_0(\zeta)$ is a Schwartz function in $(t,t^\prime) \in \R \times \R$
   supported by $\overline{\R}_+\times\overline{\R}_+$. Moreover, by virtue of Theorem
   \ref{1-8} for every $s^\prime$, $s^{\prime\prime}$ and every $M\in
   \N$ there is an $N\in \N$ such that $\parallel\!
   D^N_{\tilde{\tau}}\tilde{p}_0(\tilde{\tau},\tilde{\zeta})\!\parallel_{\mathcal{L}(H^{s^\prime,\gamma_1}(B),
  H^{s^{\prime\prime},\gamma_1-\mu+N}(B))}\leq
   c\langle\tilde{\tau},\tilde{\zeta}\rangle^{-M}$ for all $(\tilde{\tau},\tilde{\zeta})\in
   \R^{1+d}$. 
This gives us $$\parallel\!
   D^N_{\tilde{\tau}}\tilde{p}_0(t{\tau},t{\zeta})\!\parallel_{\mathcal{L}(H^{s^\prime,\gamma_1}(B),
   H^{s^{\prime\prime},\gamma_1-\mu+N}(B))}\leq
   c\langle t{\tau},t{\zeta}\rangle^{-M}\leq c\langle t\rangle^{-M}\langle \tau\rangle^{-M}$$ for $|\zeta|\geq
   \varepsilon>0$, and $t>\varepsilon>0$, for some $c>0$. It follows that the kernel of
   $g_0(\zeta)$ is also a Schwartz function in $(t,t^\prime)\in \R\times \R$ supported by
   $\overline{\R}_+\times \overline{\R}_+$. Moreover, it is $C^\infty$ in $\zeta$; in
   addition we have \begin{equation}\label{1-22}g_0(\lambda \zeta)=\kappa_\lambda g_0(\zeta)\kappa^{-1}_\lambda
   \quad \quad \textup{for}\,\, \lambda\geq 1,\,\, |\zeta|\geq
   C\end{equation} for some $C>0$ which is an evident consequence of \eqref{1-21}.
   To complete the proof we apply a tensor product expansion
   $$\tilde{q}(t,t^\prime,\tilde{\tau},\tilde{\zeta})=\sum^\infty_{j=0}c_j\varphi_j(t)\tilde{p}_j(\tilde{\tau},\tilde{\zeta})\epsilon '(t')$$
   for constants $c_j\in \C$, $\sum\vert c_j\vert<\infty$, $\varphi_j\in
   C_0^\infty([0,R)_0)$ for some $R>0$ (the latter space means the set of all $\varphi\in C_0^\infty(\overline{\R}_+)$ such that $\varphi(t)=0$
   for $t\geq R$), $\tilde{p}_j(\tilde{\tau},\tilde{\zeta})\in
   L^m(B,\boldsymbol{g};\R^{1+d}_{\tilde{\tau},\tilde{\zeta}})$,
   tending to $0$ in the respective spaces as $j\to \infty$.
   Multiplying \eqref{1-12} on both sides by $\vartheta(\zeta)$ it
   follows that $\vartheta(\zeta)g(\zeta)u=\sum_{j=0}^\infty
   c_jg_j(\zeta)u$ for $$g_j(\zeta)u=\vartheta(\zeta)\int\!\!\int e^{i(t-t^\prime)\tau}
   f_N(t,t^\prime,\zeta)\varphi_j(t)\epsilon '(t')D_{\tilde{\tau}}^N\tilde{p}_j(t\tau,t\zeta)u(t^\prime)dt^\prime\dbar\tau,$$
   obtained by integration by parts in every summand. We decompose
   $\varphi_j$ and $\epsilon '$ by applying the Taylor formula,
   namely, $$\varphi_j(t)=\sum_{l=0}^L\frac{1}{l!}t^l(\partial^l_t \varphi_j)(0)+t^{L+1}\varphi_{j,L}(t),\,
   \varphi_{j,L}(t)=\frac{1}{L!}\int_0^1(1-\theta)^L(\partial_t^{L+1}\varphi_j)(\theta
   t)d\theta$$ and $\epsilon '(t')=1+(t')^{L+1}\epsilon '_L(t'),\epsilon '_L(t')=1/L!\int_0^1(1-\theta)^L((\partial^{L+1}_{t'}\epsilon ')(\theta t')d\theta.$
    This gives us
   \begin{multline}\label{1-23}
   g_j(\zeta)u=\vartheta(\zeta)\int\!\!\int e^{i(t-t^\prime)\zeta}
   f_N(t,t^\prime,\zeta)\{\sum_{l=0}^L \alpha_{jl}t^l+\varphi_{j,L}(t)t^{L+1}\} \\\{1+(t')^{L+1}\epsilon '_{L}(t')\}
  (D_{\tilde{\tau}}^N\tilde{p}_j(t\tau,t\zeta))u(t^\prime)dt^\prime\dbar\tau
   \end{multline} for coefficients
   $\alpha_{jl}=\frac{1}{l!}(\partial_t^l\varphi_j)(0),$
   tending to zero as $j\to \infty$. Now the contributions
   $g_{j,l}(\zeta)$ in \eqref{1-23} containing $t^l$ for
   $0\leq l\leq L$ have the property $$g_{j,l}(\lambda\zeta)=\lambda^{-l}\kappa_\lambda g_{j,l}(\zeta)\kappa_\lambda^{-1}$$
   for $\vert\zeta\vert\geq C$, $\lambda\leq 1$, and hence those
   belong to
   $R_G^{-l}(\R^d)_{\mathcal{O}}$.
   They tend to zero for $j\to
   \infty$ in the space of such Green symbols. The remaining contributions to \eqref{1-23} will be
   denoted by $g_{j,l,L}(\zeta)\epsilon '_L$, $\varphi _{j,L}g_{j,L}(\zeta)$ and
   $\varphi _{j,L}g_{j,L,L}\epsilon '_L$ for $0\leq l\leq L$ (with obvious meaning of
   notation). Thus $g_j(\zeta)u=\sum_{l=0}^Lg_{j,l}(\zeta)u+
   \sum_{l=0}^L g_{j,l,L}(\zeta)\epsilon _L u+
   \varphi_{j,L}g_{j,L}(\zeta)u+\varphi_{j,L}g_{j,L,L}(\zeta)\epsilon '_L u.$
   The terms $g_{j,l,L}(\zeta)$, $g_{j,L}(\zeta)$ and
   $g_{j,L,L}(\zeta)$ represent elements of
   $R_G^{-(l+L+1)}(\R^d)_{\mathcal{O}}$
   of corresponding homogeneities for large $\vert\zeta\vert$.
   However, in the expression for $g_j$ they are multiplied by the
   functions $\varphi_{j,L}$, etc. The expressions tend to zero as $j\to
   \infty$. These multiplications do not represent classical
   symbols, but general operator-valued symbols of order $0$.
   Combined with the factors $g_{j,l,L}(\zeta),$ etc., those give
   rise
   to general symbols of order $-(l+L+1)$. Since $L$ is arbitrary
   we obtain altogether an asymptotic expansion of $g(\zeta)$ into
   homogeneous components of the desired quality.
\end{proof}
\section{Growth comparing symbols}
\begin{Thm}\label{1-18}
Let $\omega(t)$ be a cut-off function and
$\omega_\zeta(t):=\omega(t[\zeta])$, and fix $\gamma_1\in \R$.
\item[\textup{(i)}] Given $s \in \R$ for every $L\in \R$
there is an $e=e(L)$ and an $s^\prime\in \R,$ and also for every $e\in \R$
there is an $L=L(e)\in \R$ and an $s^\prime \in \R$, such that
$$(1-\omega_\zeta)t^{-L}\in S^L_{\textup{cl}}(\R^d;\mathcal{H}^{s,(\gamma_1,\gamma_2)}(B^\wedge),\mathcal{K}^{s^\prime,(\gamma_1,\gamma_2^\prime);e}(B^\wedge))$$
uniformly in compact intervals with respect to $\gamma_1, \gamma_2$ and for all $\gamma_2^\prime$.
\item[\textup{(ii)}]Given
$s \in \R$ for every $e\in \R$ there exists an
$L=L(e)$ and an $s^\prime\in \R$ and also for every $L\in \R$ there is an $e=e(L)\in \R$ and
an $s^\prime\in \R$ such that
$$(1-\omega_\zeta)t^{-L}\in S_{\textup{cl}}^L(\R^d;\mathcal{K}^{s,(\gamma_1,\gamma_2);e}(B^\wedge),
\mathcal{H}^{s^\prime,(\gamma_1,\gamma_2^\prime)}(B^\wedge)),$$ uniformly in compact intervals with respect to the
involved weights $\gamma_1$, $\gamma^\prime_2$ and for all $\gamma_2$.
\item[\textup{(iii)}] For every
$s,\gamma_2^\prime,e^\prime, L\in \R$, there is an $s^\prime\in \R$ such that
$\omega_\zeta t^L\in S_{\textup{cl}}^{-L}(\R^d;\mathcal{H}^{s,(\gamma_1,\gamma^\prime_2)}(B^\wedge),\mathcal{K}^{s^\prime,(\gamma_1,\gamma_2^\prime+L);e^\prime}(B^\wedge)).$
 \item[\textup{(iv)}] For every $s,\gamma_2,e ,L\in
\R$ there is an $s^\prime\in \R$ such that $\omega_\zeta t^L\in
S_{\textup{cl}}^{-L}(\R^d; \mathcal{K}^{s,(\gamma_1,\gamma_2);e}(B^\wedge),
\mathcal{H}^{s^\prime,(\gamma_1,\gamma_2)}(B^\wedge)).$
\end{Thm}
\begin{proof}
Let us first prepare some observations on the involved spaces $\mathcal{H}^{s,(\gamma_1,\gamma_2)}(B^\wedge)$
and $\mathcal{K}^{s,(\gamma_1,\gamma_2);e}(B^\wedge)$, respectively. The spaces can be written as non-direct sums
\[\mathcal{H}^{s,(\gamma_1,\gamma_2)}(B^\wedge)=E^s_{\textup{reg}}+E^s_{\textup{sing}},\,\,
\mathcal{K}^{s,(\gamma_1,\gamma_2);e}(B^\wedge)=F^{s;e}_{\textup{reg}}+F^{s;e}_{\textup{sing}}\]
where $E^s_{\textup{reg}}:=(1-\omega)\mathcal{H}^{s,\gamma_2}((2\B)^\wedge)$, while $E^s_{\textup{sing}}$ is the completion of $C^\infty_0(\R_+\times(B_\omega\setminus Y))$
where $B_\omega$ is the support of $\omega(r)$ as a function on $B$ (which is equal to $1$ close to $Y$) with respect to the norm
\begin{equation}\label{1-93}
\big\{\sum_{j=1}^N\|\varphi_j \omega u\|^2_{\mathcal{W}^{s,\gamma_2}(\R_+\times \R^q;\mathcal{K}^{s,\gamma_1}(X^\wedge))}\big\}^{\frac{1}{2}},
\end{equation} cf. Definition \ref{2-8}. Moreover, set $F^s_{\textup{reg}}=(1-\omega)H^s_{\textup{cone}}((2\B)^\wedge),$
and define $F^s_{\textup{sing}}$ to be the completion of $C^\infty_0(\R_+\times(B_\omega\setminus Y))$ with respect to the norm
\begin{equation}\label{1-94}
\big\{\sum_{j=1}^N\|\delta_{[t]}^{-1}(\omega\varphi_j u)\|^2_{\mathcal{W}^s(\R\times\R^q,\mathcal{K}^{s,\gamma_1}(X^\wedge))}\big\}^{\frac{1}{2}},
\end{equation} cf. Definition \ref{1-sp} and Remark \ref{1-speq}, and let
$F_{\textup{reg}}^{s;e}:=\langle t\rangle^{-e}F_{\textup{reg}}^s,\, F_{\textup{sing}}^{s;e}:=\langle t\rangle^{-e}F_{\textup{sing}}^s.$ We intend
to reduce the proof of Theorem \ref{1-18} to Theorem \ref{1-66} for the reg-spaces, and to an analogous
result for spaces $E^s$ and $F^{s;e}$, respectively, where $E^s$ means the completion of $C_0^\infty(\R_+\times (B_\omega\setminus Y))$
with respect to the norm
\begin{equation}\label{2-10}
\big\{\sum_{j=1}^N\|\varphi_j \omega u\|^2_{\mathcal{W}_1^{s,\gamma_2}(\R_+\times \R^q;\mathcal{K}^{s,\gamma_1}(X^\wedge))}\big\}^{\frac{1}{2}},
\end{equation}
cf. \eqref{eins}, while $F^{s;e}$ is the completion of $C_0^\infty(\R_+\times (B_\omega\setminus Y))$ with respect to the norm
\begin{equation}\label{2-11}
\big\{\sum_{j=1}^N\|\delta_{[t]}^{-1}\omega\varphi_j u\|^2_{\mathcal{W}_1^{s}(\R\times \R^q;\mathcal{K}^{s,\gamma_1;e}(X^\wedge))}\big\}^{\frac{1}{2}}.
\end{equation}
Observe that the relations
$a_{\textup{reg}}\in S^L(\R^d;E^s_{\textup{reg}},F^{s^\prime;e}_{\textup{reg}}), a_{\textup{sing}}\in S^L(\R^d;E^s_{\textup{sing}},F^{s^\prime;e}_{\textup{sing}})$ and
$a_{\textup{reg}}|_{E^s_{\textup{reg}}\cap E^s_{\textup{sing}}}=a_{\textup{sing}}|_{E^s_{\textup{reg}}\cap E^s_{\textup{sing}}}$
imply
\begin{equation}\label{2-22}a_{\textup{reg}}+a_{\textup{sing}}\in S^L(\R^d;E^s_{\textup{reg}}+E^s_{\textup{sing}},F_{\textup{reg}}^{s^\prime;e}+F_{\textup{sing}}^{s^\prime;e})\end{equation}
with the right weight in $F_{\textup{reg}}^{s^\prime;e}$ and $F_{\textup{sing}}^{s^\prime;e}$, cf. Theorem \ref{1-18} (i),
and similarly for (ii), (iii), (iv). From Theorem \ref{1-66} we have the desired relation for the reg-spaces and $s^\prime=s$.
So we have to concentrate on the sing-spaces. Based on Lemma \ref{1-92} we have the following continuous embeddings
\[E^{s+M}\hookrightarrow E_{\textup{sing}}^{s}\hookrightarrow E^{s-M},\,\,
E_{\textup{sing}}^{s+M}\hookrightarrow E^{s}\hookrightarrow E_{\textup{sing}}^{s-M},\]
\[F^{s+M;e}\hookrightarrow F_{\textup{sing}}^{s}\hookrightarrow F^{s-M;e},\,\,
F_{\textup{sing}}^{s+M;e}\hookrightarrow F^{s;e}\hookrightarrow F_{\textup{sing}}^{s-M;e}\]
which follow by comparing
\eqref{2-10} and \eqref{2-11} with \eqref{1-93} and \eqref{1-94}, respectively.
We shall prove the desired relation of Theorem \ref{1-18} in the following form, say,
first for (i). Given $s\in \R$ for every $L\in \R$ there is an $e=e(L)$
such that
\begin{equation}\label{2-13}
(1-\omega_\zeta)t^{-L}\in S_{\textup{cl}}^L(\R^d;E^s,F^{s;e})
\end{equation}(remember that the weights as in (i) are hidden in the meaning of the spaces $E^s$, $F^{s;e}$).
The group actions in the spaces $E^s$, $F^{s;e}$ with respect to the $t$-direction are given by the same
expression as for $E^s_{\textup{sing}}$, $F^{s;e}_{\textup{sing}}$, namely,
$(\chi_\lambda u)(t,\cdot)=\lambda^{\frac{b+1}{2}} u(\lambda t,\cdot), \lambda\in \R_+.$
We will show that the operator $(1-\omega_\zeta)t^{-L}$ defines a continuous map
$$(1-\omega_\zeta)t^{-L}:E^s\to F^{s;e},$$
for any fixed $\zeta\in \R^d.$ It is also smooth in $\zeta$. Then because of
\[\|(1-\omega_\zeta)t^{-L}\|_{\mathcal{L}(E^{s+M}_{\textup{sing}},F^{s-M;e}_{\textup{sing}})}\leq c\|(1-\omega_\zeta)t^{-L}\|_{\mathcal{L}(E^{s},F^{s;e})}\]
for a constant $c>0$ which is a consequence of the above-mentioned continuous embeddings
$E_{\textup{sing}}^{s+M}\hookrightarrow E^{s}, F^{s;e}\hookrightarrow F_{\textup{sing}}^{s-M;e},$
we obtain the same of the family of maps
\[(1-\omega_\zeta)t^{-L}:E_{\textup{sing}}^{s+M}\to F_{\textup{sing}}^{s-M;e}.\]
By virtue of
$(1-\omega_{\lambda\zeta})t^{-L}=\chi_\lambda(1-\omega_\zeta)t^{-L}\chi_\lambda^{-1}$
for all $\lambda\geq 1$, $|\zeta|\geq \textup{const}$, it follows that
$(1-\omega_\zeta)t^{-L}\in S_{\textup{cl}}^L(\R^d;E_{\textup{sing}}^{s+M},F_{\textup{sing}}^{s-M;e}),$
or, alternatively, since this is true of all $s$,
\begin{equation}\label{2-14}
(1-\omega_\zeta)t^{-L}\in S^L(\R^d;E_{\textup{sing}}^{s},F_{\textup{sing}}^{s-2M;e})).
\end{equation}
Moreover, by Theorem \ref{2-34} we already have
$(1-\omega_\zeta)t^{-L}\in S^L(\R^d;E^s_{\textup{reg}},F^{s;e}_{\textup{reg}})$
which entails
\begin{equation}\label{2-15}(1-\omega_\zeta)t^{-L}\in S^L(\R^d;E^s_{\textup{reg}},F^{s-2M;e}_{\textup{reg}}).\end{equation}
Then \eqref{2-14} and \eqref{2-15} yield
\begin{equation}\label{2-35}(1-\omega_\zeta)t^{-L}\in S^L(\R^d;E^s_{\textup{reg}}+E^s_{\textup{sing}},F^{s-2M;e}_{\textup{reg}}+F^{s-2M;e}_{\textup{sing}}),\end{equation}
cf. \eqref{2-22}, i.e. the desired relations as in (i). To complete the proof of (i)
we first verify the desired property for $s\in \N$. Here we employ the following analogues of the
characterisation of the spaces involved in $E_{\textup{sing}}^{s}$, $F_{\textup{sing}}^{s;e}$. The space $E_{\textup{sing}}^{s}$,
$s\in \N$, relies on \eqref{2-10}, the analogue of the norm \eqref{1-93} where
$\mathcal{W}^{s,\gamma_2}(\R_+\times \R^q;\mathcal{K}^{s,\gamma_1}(X^\wedge))$ is replaced by $\mathcal{W}_1^{s,\gamma_2}(\R_+\times \R^q;\mathcal{K}^{s,\gamma_1}(X^\wedge))$,
cf. Lemma \ref{2-16}. Now, instead of Proposition \ref{2-17} we have for every $s\in \N$ the equivalence of norms
\begin{align*}
&\| u\|_{\mathcal{W}_1^{s,\gamma_2}(\R_+\times \R^q,\mathcal{K}^{s,\gamma_1}(X^\wedge))} \sim\big\{
 \| u\|^2_{\mathcal{W}_1^{0,\gamma_2}(\R_+\times \R^q,\mathcal{K}^{0,\gamma_1}(X^\wedge))}\notag\\
&+\sum_{|\beta|=s}\|(t\partial_t)^{\beta^\prime} D_y^{\beta^{\prime\prime}}u\|^2_{\mathcal{W}_1^{0,\gamma_2}(\R_+\times \R^q,
\mathcal{K}^{0,\gamma_1}(X^\wedge))}\notag\\&+\sum_{|\alpha|=s, 1\leq l\leq N}\| D_l^\alpha u\|^2_{\mathcal{W}_1^{0,\gamma_2}(\R_+\times \R^q,\mathcal{K}^{0,\gamma_1;-s}(X^\wedge))}\notag\\&+\sum_{|\alpha|=s, 1\leq l\leq N,|\beta|=s}
\|(t\partial_t)^{\beta^\prime} D_y^{\beta^{\prime\prime}}D_l^\alpha u\|^2_{\mathcal{W}_1^{0,\gamma_2}(\R_+\times \R^q,
\mathcal{K}^{0,\gamma_1;-s}(X^\wedge))}\big\}^{1/2}
\end{align*}for the operators $D_l^{\alpha}$ from Proposition \textup{\ref{1-79}} or, if desired, a corresponding sum over $0<|\alpha|\leq s$,
for $\mathcal{K}^{0,\gamma_1;-|\alpha|}(X^\wedge)$, and $0<|\beta|\leq s$. The space $F_{\textup{sing}}^{s}$, $s\in \N$, relies
on \eqref{2-11}, the analogue of the norm \eqref{1-94} where
$\mathcal{W}^s(\R\times \R^q,\mathcal{K}^{s,\gamma_1}(X^\wedge))$ is replaced by $\mathcal{W}_1^s(\R\times \R^q,\mathcal{K}^{s,\gamma_1}(X^\wedge))$.
Then, by virtue of the corresponding analogue of Proposition \ref{2-18} for $s\in \N$ we have the equivalence of norms
\begin{align*}
&\|u\|_{\mathcal{W}_1^{s,\gamma_1}(\R\times\R^q\times X^\wedge)}
\sim\{\big\|u(t,r[t]^{-1},x,y[t]^{-1})\|^2_{\mathcal{W}_1^0(\R\times \R^q,\mathcal{K}^{0,\gamma_1}(X^\wedge))}\notag
\\ &+\sum_{|\beta|=s}\|D_t^{\beta^\prime}D_y^{\beta^{\prime\prime}}
u(t,r[t]^{-1},x,y[t]^{-1})\|^2_{\mathcal{W}_1^0(\R\times \R^q,\mathcal{K}^{0,\gamma_1}(X^\wedge))}\notag
\\ &+\sum_{|\alpha|=s,1\leq j\leq N}\|D_j^\alpha
 u(t,r[t]^{-1},x,y[t]^{-1})\|^2_{\mathcal{W}_1^0(\R\times \R^q,\mathcal{K}^{0,\gamma_1;-s}(X^\wedge))}\notag
 \\&+\sum_{|\alpha|=s,1\leq j\leq N,|\beta|=s}\|D_t^{\beta^\prime}D_y^{\beta^{\prime\prime}}D_j^\alpha
 u(t,r[t]^{-1},x,y[t]^{-1})\|^2_{\mathcal{W}_1^0(\R\times \R^q,\mathcal{K}^{0,\gamma_1;-s}(X^\wedge))}\big\}^{\frac{1}{2}}
\end{align*} again with the operators $D_l^\alpha$ in Proposition \ref{1-79}, or, if desired, a corresponding sum over $0<|\alpha|\leq s$.
Now, according to the requirement of (i) we have to show the following. Given $s\in \N$
for every $L$ there is an $e=e(L)$ such that for arbitrary
$u\in C_0^\infty(\R_+\times \R^q,\mathcal{K}^{0,\gamma_1}(X^\wedge))$ vanishing in a neighbourhood of $t=0$
and which are localised in the expression \eqref{1-speqe}
and \eqref{2-21}, i.e. with included factors $\omega\varphi_j$, we have
 \begin{equation}\label{2-19}
 \begin{split}
 \|(1-\omega_\zeta)t^{-L}u(t,r[t]^{-1},& \,x, y[t]^{-1})\|_{{\langle t\rangle^{-e}}\mathcal{W}_1^0(\R\times \R^q,\mathcal{K}^{0,\gamma_1}(X^\wedge))}\\
 &\leq c\|u(t,r,x,y)\|_{\mathcal{W}_1^{0,\gamma_2}(\R_+\times \R^q,\mathcal{K}^{0,\gamma_1}(X^\wedge))}
 \end{split}\end{equation} ($t^{-\gamma_2^\prime}\mathcal{W}_1^0$ makes sense because $u$ is supported in $\R_{+,t}$). Moreover,
\begin{equation}\label{2-23}
\|D_t^{\tilde{\beta}^\prime}D_y^{\tilde{\beta}^{\prime\prime}}
(1-\omega_\zeta)t^{-L}u(t,r[t]^{-1},x,y[t]^{-1})\|_{\langle t\rangle^{-e}\mathcal{W}_1^0(\R\times \R^q,\mathcal{K}^{0,\gamma_1}(X^\wedge))}
\leq c N_1(u)\end{equation}
for 
\begin{align*}
&N_1(u):=\big\{\sum_{|\beta|=s}\|(t\partial_t)^{\beta^\prime} D_y^{\beta^{\prime\prime}}u(t,r,x,y)\|^2_{\mathcal{W}_1^{0,\gamma_2}(\R_+\times \R^q,
\mathcal{K}^{0,\gamma_1}(X^\wedge))}\notag \\&+\sum_{|\alpha|=s, 1\leq l\leq N}\| D_l^\alpha u(t,r,x,y)\|^2_{\mathcal{W}_1^{0,\gamma_2}(\R_+\times \R^q,
\mathcal{K}^{0,\gamma_1;-s}(X^\wedge))}\notag\\&+\sum_{|\alpha|=s, 1\leq l\leq N,|\beta|=s}
\|(t\partial_t)^{\beta^\prime} D_y^{\beta^{\prime\prime}}D_l^\alpha u(t,r,x,y)\|^2_{\mathcal{W}_1^{0,\gamma_2}(\R_+\times \R^q,
\mathcal{K}^{0,\gamma_1;-s}(X^\wedge))}\big\}^{\frac{1}{2}}\end{align*} for $(\tilde{\beta}^\prime,\tilde{\beta}^{\prime\prime})$
such that $|\tilde{\beta}^\prime|+|\tilde{\beta}^{\prime\prime}|=s$ and
\begin{equation}\label{2-24}
\|D_j^{\tilde{\alpha}}(1-\omega_\zeta)t^{-L}u(t,r[t]^{-1},x,y[t]^{-1})\|_{\langle t\rangle^{-e}\mathcal{W}_1^0(\R\times \R^q,\mathcal{K}^{0,\gamma_1;-s}(X^\wedge))}\leq cN_1(u)\end{equation}
for $\tilde{\alpha}$ such that $|\tilde{\alpha}|=s$, $0\leq j\leq N$,
for all $\zeta$ in any bounded set in $\R^q$, for a constant $c>0$ independent on $u$,
and finally,
\begin{equation}\label{2-27}
\|D_t^{\tilde{\beta}^\prime}D_y^{\tilde{\beta}^{\prime\prime}}D_j^\alpha(1-\omega_\zeta)t^{-L}
 u(t,r[t]^{-1},x,y[t]^{-1})\|_{\mathcal{W}_1^0(\R\times \R^q,\mathcal{K}^{0,\gamma_1;-s}(X^\wedge))}\leq cN_1(u).\end{equation}
Let us start with the proof of estimate \eqref{2-19}. For abbreviation we write $\omega$ rather than $\omega_\zeta$;
the parameter $\zeta$ is fixed in the computation. For convenience the computation will be carried out for $q=1$;
however, we keep writing $\R^q$ in the corresponding norm expression. In addition for simplicity from now on we consider the case $\textup{dim}\,X=0$;
the consideration in general is completely analogous. The norm of a function $v(t,r,y)$ in the space $\mathcal{W}_1^{0,\gamma_2}(\R_+\times \R^q,
\mathcal{K}^{0,\gamma_1-s})$ for $\mathcal{K}^{0,\gamma_1}:=\mathcal{K}^{0,\gamma_1}(\R_+)$ is equivalent to the one of $t^{-\gamma_2}L^2(\R_+\times \R^q,
\mathcal{K}^{s,\gamma_1})$, i.e. writing equality when we mean equivalence, we have
\begin{multline}\label{2-29}\|v(t,r,y)\|^2_{\mathcal{W}^{0,\gamma_2}_1(\R_+\times \R^q,\mathcal{K}^{0,\gamma_1})}=\|t^{-\gamma_2}t^{(1+q)/2}v(t,r,y)\|^2_{L^2(\R_+\times \R^q,\mathcal{K}^{0,\gamma_1})}\\=\int_{\R^q}\int_0^\infty\int_0^\infty|r^{-\gamma_1}\langle r\rangle^{\gamma_1}t^{-\gamma_2}v(t,r,y)|^2t^{1+q}drdtdy,
\end{multline}
cf. the second term on the right of \eqref{2-20} for $u=0$. Here we used the identification
$\mathcal{K}^{0,\gamma_1}(\R_+)=r^{\gamma_1}\langle r\rangle^{-\gamma_1}L^2(\R_+).$ Moreover,
for a function $v(t,r,y)$ in the space $\langle t\rangle^{-e}\mathcal{W}_1^0(\R\times \R^q,\mathcal{K}^{0,\gamma_1})$ we have
\begin{multline}\label{2-25}\|v(t,r,y)\|^2_{\langle t\rangle^{-e}\mathcal{W}^{0}_1(\R\times \R^q,\mathcal{K}^{0,\gamma_1})}=\|\langle t\rangle^{e}v(t,r,y)\|^2_{L^2(\R\times \R^q,\mathcal{K}^{0,\gamma_1})} \\=\int\!\!\!\int\!\!\!\int|\langle t\rangle^{e}r^{-\gamma_1}\langle r\rangle^{\gamma_1}v(t,r,y)|^2t^{1+q}drdtdy.
\end{multline}
As noted before we assume the functions to be supported in $t$ on a half-axis $[\varepsilon,\infty)$ for some $\varepsilon>0$.
Thus for the proof of \eqref{2-19} we observe that
\begin{align}\label{2-26}
&\int\!\!\!\int\!\!\!\int|\langle t\rangle^{e}r^{-\gamma_1}\langle r\rangle^{\gamma_1}(1-\omega_\zeta)t^{-L}u(t,r[t]^{-1},y[t]^{-1})|^2t^{1+q}drdtdy\notag\\
&\leq c\int\!\!\!\int\!\!\!\int|\langle t\rangle^{e}(r[t])^{-\gamma_1}\langle r[t]\rangle^{\gamma_1}(1-\omega_\zeta)t^{-L}u(t,r,y)|^2t^{2(1+q)}drdtdy\notag\\
&\leq c\int\!\!\!\int\!\!\!\int|\langle t\rangle^e[t]^{-\gamma_1}r^{-\gamma_1}\langle r\rangle^{\gamma_1}\langle[t]\rangle^{\gamma_1}
(1-\omega_\zeta)t^{-L}u(t,r[t]^{-1},y[t]^{-1})|^2t^{2(1+q)}drdtdy\notag\\
&\leq c\int\!\!\!\int\!\!\!\int|\varphi(t)t^{-\gamma_2}r^{-\gamma_1}\langle r\rangle^{\gamma_1}u(t,r,y)|^2t^{(1+q)}drdtdy
\end{align} for $\varphi(t)=t^{\gamma_2}\langle t\rangle^e\langle t\rangle^{-L}[t]^{-\gamma_1}\langle[t]\rangle^{\gamma_1}t^{1+q}(1-\omega_\zeta)\leq \textup{const.}$,
as soon as $e$ is chosen in a suitable way.
For the proof of \eqref{2-23} we apply the norm expression \eqref{2-25} to $f=D_t^{\beta^\prime}D_y^{\beta^{\prime\prime}}(1-\omega_\zeta)t^{-L}
u(t,r[t]^{-1},y[t]^{-1})$, let $v:=(1-\omega_\zeta)t^{-L}
u(t,r[t]^{-1},y[t]^{-1})$ and check the case $\tilde{\beta}^\prime=1$ and $\tilde{\beta}^{\prime\prime}=0$. In other words we take the functions
\begin{multline}\label{2-28}
\partial_t v(t,r[t]^{-1},y[t]^{-1})=v_t(t,r[t]^{-1},y[t]^{-1})\\+r\partial_t[t]^{-1}(\partial_r v)(t,r[t]^{-1},y[t]^{-1})+y\partial_t[t]^{-1}(\partial_y v)
(t,r[t]^{-1},y[t]^{-1})
\end{multline} where $v_t$ means the derivative of $v$ in the first $t$-variable. In the estimate it suffices to consider the summands separately.
The consideration for the first summand on the right of \eqref{2-28} is nearly the same as for \eqref{2-26}. In fact, the contribution where only $(1-\omega_\zeta)t^{-L}$
is differentiated in $t$ is exactly as what we did in \eqref{2-26}, while for the term containing $u_t$ we take first norm expression
on the right of \eqref{2-23} for $\beta^\prime=1$, $\beta^{\prime\prime}=0$. For the second summand on the right of \eqref{2-28}
the only novelty is the extra factor $r\partial_t[t]^{-1}$. Besides a term with the $t$-derivative of $(1-\omega_\zeta)t^{-L}$
which causes no extra problem, on the left of \eqref{2-23} we have
\begin{align*}
&\int\!\!\!\int\!\!\!\int|\langle t\rangle^{e}r^{-\gamma_1}\langle r\rangle^{\gamma_1}(1-\omega_\zeta)t^{-L}\partial_tu(t,r[t]^{-1},y[t]^{-1})|^2t^{1+q}drdtdy\\
&=\int\!\!\!\int\!\!\!\int|\langle t\rangle^{e}r^{-\gamma_1}\langle r\rangle^{\gamma_1}(1-\omega_\zeta)t^{-L}r\partial_t[t]^{-1}(\partial_ru)(t,r[t]^{-1},y[t]^{-1})|^2t^{1+q}drdtdy\\
&\leq c\int\!\!\!\int\!\!\!\int|\varphi_1(t)t^{-\gamma_2}\langle r\rangle^{-1} r^{-\gamma_1}\langle r\rangle^{\gamma_1}u(t,r,y)|^2t^{(1+q)}drdtdy\\
&\leq c\|D^1 u\|^2_{\mathcal{W}_1^{0,\gamma_2}(\R_+\times \R^q;\mathcal{K}^{0,\gamma_1;-1})},
\end{align*}
here $D^1=r\partial_r$, because we took the $\R_+$-case, $\varphi_1(t)=t^{\gamma_2}\langle t\rangle^e\langle t\rangle^{-L}|\partial_t[t]^{-1}|[t]^{-\gamma_1}\langle[t]\rangle^{\gamma_1}t^{1+q}(1\newline -\omega_\zeta)\leq const$ for a suitable choice of $e$.
On the respective right hand side we took a norm expression
from the third sum for $s=1$, here for $|\alpha|=1$. The $\langle r\rangle^{-1}$-power
belonging to the $\mathcal{K}^{0,\gamma_1;-1}(\R_+)$-norm does not
cause any problem since the function is localised near $r=0$ by the involved cut-off factor $\omega(r)$.
The third term on the right of \eqref{2-28} is easy; here it suffices to refer to the first norm expression on
the right of \eqref{2-23} for $\beta^\prime=0$, $\beta^{\prime\prime}=1$, and we take into account the fact that the
function contains a factor in $y$ of compact support. The higher derivatives can be treated in an analogous manner.\\
\\
The estimate \eqref{2-24} may be checked first for $|\tilde{\alpha}|=1$; in this case we take the second term on the
right of \eqref{2-24} and proceed as before, then analogous estimates hold for arbitrary $\tilde{\alpha}$.
For \eqref{2-27} we can take again first order derivative on the left and in the estimate we refer to the third
term on the right of \eqref{2-27}. Also here we drop the elementary consideration; the higher derivatives can be treated
in an analogous manner. Summing up we proved the first direction of Theorem \ref{1-18} (i) for $s\in \N$. The converse direction, i.e.
to find $L=L(e)$ for given $e$ is clear from the structure of the functions $\varphi(t)$, $\varphi_1(t)$, etc.\\
\\
Let us now pass to the proof of Theorem \ref{1-18} (ii) for $s\in \N$. In the case $s=0$ similarly as \eqref{2-19}
we have to show
\[\|(1-\omega_\zeta)t^{-L}u(t,r,y)\|_{\mathcal{W}_1^{0,\gamma^\prime_2}(\R_+\times \R^q,\mathcal{K}^{0,\gamma_1}(X^\wedge))}
\leq c \|u(t,r[t]^{-1},y[t]^{-1})\|_{{\langle t\rangle^{-e}}\mathcal{W}_1^0(\R\times \R^q,\mathcal{K}^{0,\gamma_1}(X^\wedge))}.\]
Using \eqref{2-29} and \eqref{2-25} (again in the $\R_+$-case and computations for $q=1$) and the estimate
\[\langle r[t]^{-1}\rangle^{\gamma_1}\leq c\langle r\rangle^{\gamma_1}\pi(t,\gamma_1), \qquad \pi(t;\gamma_1):=\bigg\{\begin{array}{ll}
 [t]^{-\gamma_1}&\textup{for}\quad \gamma_1\leq 0,\\ 1&\textup{for}\quad \gamma_1\geq 0\end{array}\] we obtain
 \begin{align*}
 &\int\!\!\!\int\!\!\!\int|(1-\omega_\zeta)t^{-L}r^{-\gamma_1}\langle r\rangle^{\gamma_1}t^{-\gamma^\prime_2}u(t,r,y)|^2t^{1+q}drdtdy\\
 &=\int\!\!\!\int\!\!\!\int|(1-\omega_\zeta)t^{-L}(\tilde{r}[t]^{-1})^{-\gamma_1}
 \langle\tilde{r}[t]^{-1}\rangle^{\gamma_1}t^{-\gamma^\prime_2}u(t,\tilde{r}[t]^{-1},\tilde{y}[t]^{-1})|^2[t]^{-(1+q)}t^{1+q}d\tilde{r}dtd\tilde{y}\\
 &\leq c\int\!\!\!\int\!\!\!\int|\varphi_2(t)\langle t\rangle^{e}r^{-\gamma_1}\langle r\rangle^{\gamma_1}u(t,r[t]^{-1},x,y[t]^{-1})|^2t^{1+q}drdtdy
 \end{align*}
for $\varphi_2=(1-\omega_\zeta)t^{-L}t^{-\gamma_2^\prime}[t]^{-(1+q)}\langle t\rangle^{-e}[t]^{\gamma_1}\pi(t;\gamma_1)$. It is now clear that
for every $e\in \R$ there is an $L\in \R$ such that $|\varphi_2(t)|\leq const$ and also for every $L\in \R$ there is an $e\in \R$ such
that $|\varphi_2(t)|\leq const.$ This is uniform in compact $\gamma_1$- and $\gamma_2^\prime$-intervals. The next point in proving Theorem \ref{1-18} (ii) is
to show the estimates
\begin{equation}\label{2-30}
\|(t\partial_t)^{\tilde{\beta}^\prime} D_y^{\tilde{\beta}^{\prime\prime}}(1-\omega_\zeta)t^{-L}u(t,r,y)\|_{\mathcal{W}_1^{0,\gamma_2^\prime}(\R_+\times \R^q,
\mathcal{K}^{0,\gamma_1}(X^\wedge))}\leq c N_2(u)\end{equation}
for
\begin{align*}
&N_2(u):= \big\{\sum_{|\beta|=s}\|D_t^{\beta^\prime}D_y^{\beta^{\prime\prime}}
u(t,r[t]^{-1},x,y[t]^{-1})\|^2_{\langle t\rangle^{-e}\mathcal{W}_1^0(\R\times \R^q,\mathcal{K}^{0,\gamma_1}(X^\wedge))}
\notag\\ &+\sum_{|\alpha|=s,1\leq j\leq N}\|D_j^\alpha
 u(t,r[t]^{-1},x,y[t]^{-1})\|^2_{\langle t\rangle^{-e}\mathcal{W}_1^0(\R\times \R^q,\mathcal{K}^{0,\gamma_1;-s}(X^\wedge))}
\notag \\&+\sum_{|\alpha|=s,1\leq j\leq N,|\beta|=s}\|D_t^{\beta^\prime}D_y^{\beta^{\prime\prime}}D_j^\alpha
 u(t,r[t]^{-1},x,y[t]^{-1})\|^2_{\langle t\rangle^{-e}\mathcal{W}_1^0(\R\times \R^q,\mathcal{K}^{0,\gamma_1;-s}(X^\wedge))}\big\}^{\frac{1}{2}},
\end{align*} as well as,
\begin{equation}\label{2-31}
\|D_l^{\tilde{\alpha}}(1-\omega_\zeta)t^{-L}u(t,r,y)\|_{\mathcal{W}_1^{0,\gamma_2^\prime}(\R_+\times \R^q,
\mathcal{K}^{0,\gamma_1;-s}(X^\wedge))}\leq c N_2(u)\end{equation}
and
\begin{equation}\label{2-32}
\|(t\partial_t)^{\tilde{\beta}^\prime} D_y^{\tilde{\beta}^{\prime\prime}}D_l^{\tilde{\alpha}} (1-\omega_\zeta)t^{-L}u(t,r,y)\|_{\mathcal{W}_1^{0,\gamma_2^\prime}(\R_+\times \R^q,
\mathcal{K}^{0,\gamma_1;-s}(X^\wedge))}\leq c N_2(u).\end{equation}
Let us now verify \eqref{2-30}. Again we look at a typical case, say, $\tilde{\beta}^\prime=1$, $\tilde{\beta}^{\prime\prime}=0$.
Then for the estimate we take the first expression on the right of \eqref{2-30} for ${\beta}^\prime=1$. Thus, since there is
only the derivative of $u$ in the first $t$-variable we have
\begin{align*}
&\|(1-\omega_\zeta)t^{-L}(t\partial_t)u(t,r,y)\|^2_{\mathcal{W}_1^{0,\gamma_2^\prime}(\R_+\times \R^q,
\mathcal{K}^{0,\gamma_1}(X^\wedge))}\\
&=\int\!\!\!\int\!\!\!\int|(1-\omega_\zeta)t^{-L}r^{-\gamma_1}\langle r\rangle^{\gamma_1}t^{-\gamma^\prime_2}u_t(t,r,y)|^2t^{1+q}drdtdy\\
&=\int\!\!\!\int\!\!\!\int|(1-\omega_\zeta)t^{-L}(\tilde{r}[t]^{-1})^{-\gamma_1}\langle \tilde{r}[t]^{-1}\rangle^{\gamma_1}t^{-\gamma^\prime_2+1}u_t(t,\tilde{r}[t]^{-1},\tilde{y}[t]^{-1})|^2[t]^{-(1+q)}t^{1+q}d\tilde{r}dtd\tilde{y}.
\end{align*} Then the rest of the estimate is as before; the only difference is the extra
$t$-factor. Again we may interchange the role of $e$ and $L$. In the case \eqref{2-31} we check as a typical
case $\tilde{\alpha}=1$ where $D^{\tilde{\alpha}}_l=r\partial_r$ and it suffices to take into account the second term on
the right of \eqref{2-31} for $\alpha=1$. Thus
\begin{align}\label{2-36}
&\|(1-\omega_\zeta)t^{-L}(r\partial_r)u(t,r,y)\|^2_{\mathcal{W}_1^{0,\gamma_2^\prime}(\R_+\times \R^q,
\mathcal{K}^{0,\gamma_1;-1}(X^\wedge))}\notag\\
&=\int\!\!\!\int\!\!\!\int|(1-\omega_\zeta)t^{-L}r^{-\gamma_1}\langle r\rangle^{\gamma_1}t^{-\gamma^\prime_2}\langle r\rangle^{-1}(r\partial_r u)(t,r,y)|^2t^{1+q}drdtdy\notag\\
&=\int\!\!\!\int\!\!\!\int|(1-\omega_\zeta)t^{-L}(\tilde{r}[t]^{-1})^{-\gamma_1}\langle \tilde{r}[t]^{-1}\rangle^{\gamma_1}t^{-\gamma^\prime_2}\langle \tilde{r}[t]^{-1}\rangle^{-1}(\tilde{r}[t]^{-1})\notag\\&(\partial_r u)(t,\tilde{r}[t]^{-1},\tilde{y}[t]^{-1})|^2[t]^{-(1+q)}t^{1+q}d\tilde{r}dtd\tilde{y}\notag\\
&\leq \int|\tilde{\varphi_3}(t,r)\langle t\rangle^e r^{-\gamma_1}\langle r\rangle^{\gamma_1} (r\partial_r)u(t,r[t]^{-1},y[t]^{-1})\langle r \rangle^{-1}|^2t^{1+q}drdtdy
\end{align} for $\tilde{\varphi}_3(t,r)=(1-\omega_\zeta)t^{-L}(r[t]^{-1})^{-\gamma_1}\langle r[t]^{-1}\rangle^{\gamma_1}t^{-\gamma_2^\prime}
\langle r[t]^{-1}\rangle^{-1}(r[t]^{-1})[t]^{-(1+q)}\langle t\rangle^{-e} r^{\gamma_1}\langle r\rangle^{-\gamma_1}\newline\langle r\rangle r^{-1}.$
The right hand side of \eqref{2-36} may be estimated by an integral with
\[\varphi_3(t):=(1-\omega_\zeta)t^{-L}[t]^{\gamma_1}\pi(t,\gamma_1)t^{-\gamma_2}\pi(t,-1)[t]^{1+q}\langle t\rangle^{-e}\]
instead of $\tilde{\varphi}_3$. For \eqref{2-32} the typical case is $\tilde{\beta}^\prime=1$, $\tilde{\beta}^{\prime\prime}=0$,
$\tilde{\alpha}=1$; the other derivatives can be treated in a similar way. Here we take into account the third term on the right
of \eqref{2-32} for $\tilde{\beta}^\prime=1$, $\tilde{\beta}^{\prime\prime}=0$, $\tilde{\alpha}=1$. The consideration is a combination
of \eqref{2-36} with what we did for \eqref{2-32}. In fact, we apply \eqref{2-36} for $(t\partial_t)(1-\omega_\zeta)t^{-L}u(t,r,y)$
rather than $(1-\omega_\zeta)t^{-L}u(t,r,y)$. The $t$-derivative of $(1-\omega_\zeta)t^{-L}$ is harmless so we ignore this, and then we obtain an analogue of
the right hand side of \eqref{2-36} for $((t\partial_t)u)(t,r[t]^{-1},y[t]^{-1})$ (the $t$-derivative only acts on the first $t$-variable).
The third term on the right of \eqref{2-32} contains the other $t$-derivatives by the chain rule, but those make the
right hand side only larger but do not affect the result. So far we possess \eqref{2-13} and, analogously,
\[(1-\omega_\zeta)t^{-L}\in S_{\textup{cl}}^L(\R^d;F^{s;e},E^s)\] for $s\in \N$, for every $e$ with a resulting $L$ and vice versa,
uniformly in compact weight intervals for all the involved weights. Using the dual spaces referring to the scalar products
\eqref{1-89} and \eqref{2-33}, respectively, we have analogues of Remarks \ref{2-37} and \ref{2-38} for the spaces
with trivial group actions (indicated by subscript $1$). This yields resulting spaces with opposite smoothness, $(t,r)$-weights
at $0$ and weights for $t\to \infty$. Thus the proof of (ii) which is done for $s\in \N$ gives us by duality (i) for $-s\in \N$,
while (i) already proved for $s\in \N$ yields (ii) for $-s\in \N$. The interpolation properties of Theorem \ref{2-39} and
\ref{2-34} that are true in analogous form for the spaces with subscript $1$ yield (i) and (ii) for all $s\in \R$ for the
respective spaces. Now \eqref{2-35} for (i) and the analogue of such a relation for (ii) finally completes the proof of
Theorem \ref{1-18} (i), (ii). The properties (iii), (iv) follow again from the twisted homogeneity with respect to $\zeta$ and from the fact that the spaces in question coincide
for small $t.$ 
\end{proof}


\end{document}